\pdfoutput=1

\documentclass[a4paper]{article}    
\usepackage{amsthm,amssymb,amsmath,enumerate,graphicx}

\newtheorem{definition}{Definition}

\newtheorem{theorem}[definition]{Theorem}
\newtheorem{corollary}[definition]{Corollary}
\newtheorem{lemma}[definition]{Lemma}

\newtheorem{example}{Example}

\newcommand{\C}{{\mathcal C}}

\newcommand{\E}{{\mathcal E}}
\newcommand{\I}{{\mathcal I}}

\newcommand{\N}{{\mathbb N}}
\renewcommand{\P}{{\mathcal P}}
\newcommand{\R}{{\mathcal R}}

\renewcommand{\S}{{\mathcal S}}
\newcommand{\Cbar}{{\overline C}}

\newcommand{\Ubar}{{\overline U}}

\newcommand{\interior}{\mathaccent"7017\relax}
\newcommand{\sm}{\setminus}

\newcommand{\im}{{\rm Im\>}}

\let\eps=\varepsilon
\let\sub=\subseteq
\let\subset=\subseteq
\let\supset=\supseteq
\let\phi=\varphi
\let\es=\emptyset

\newcommand\restr{\!\restriction\!}
\newcommand{\tst}{topological spanning tree}

\newcommand{\COMMENT}[1]{}

\newcommand{\noproof}{\unskip\nobreak\hfill\penalty50\hskip2em\hbox{}\nobreak\hfill%
       $\square$\parfillskip=0pt\finalhyphendemerits=0\par}

\newcommand{\emtext}[1]{\text{\em #1}}

\newenvironment{txteq}
  {\begin{equation}\begin{minipage}[c]{0.8\textwidth}\em}
  {\end{minipage}\ignorespacesafterend\end{equation}\ignorespacesafterend}

\newcommand{\assign}{
  \mathrel{\mathop{:}}=
}

\def\lowfwd #1#2#3{{\setbox0\hbox{$#1$}\setbox1\hbox{$E'\!$}
            \mathchoice
            {{\mathop{\kern0pt #1}\limits^{\kern#2pt\raise.#3ex
     \vbox to 0pt{\hbox{$\scriptscriptstyle\rightarrow$}\vss}}}}%
            {{\mathop{\kern0pt #1}\limits^{\kern#2pt\raise.#3ex
     \vbox to 0pt{\hbox{$\scriptscriptstyle\rightarrow$}\vss}}}}%
            {\ifdim\wd0<\wd1{\,\vec{#1}\,}\else
     {\mathop{\kern0pt #1}\limits^{\kern#2pt\raise.0ex
     \vbox to 0pt{\hbox{$\scriptscriptstyle\rightarrow$}\vss}}}\fi}%
            {{\vec{#1}}}%
            }}
\def\fwd #1#2{{\lowfwd{#1}{#2}{15}}}
\def\lowbkwd #1#2#3{{\mathop{\kern0pt #1}\limits^{\kern#2pt\raise.#3ex
     \vbox to 0pt{\hbox{$\scriptscriptstyle\leftarrow$}\vss}}}}

\def\vC{\kern-1pt\fwd C3\kern-.5pt}

\def\vCC{\kern-.7pt\fwd{\C}3\kern-.7pt}
\def\vd{\kern-1pt\lowfwd d2{10}\kern-1pt}
\def\vD{\kern-.7pt\fwd D3\kern-.5pt}
\def\ve{\kern-1pt\lowfwd e{1.5}1\kern-1pt}
\def\vf{\kern-1pt\lowfwd f{1.5}1\kern-1pt}
\def\fv{\kern-1pt\lowbkwd f{1.5}1\kern-1pt}
\def\ev{\kern-1pt\lowbkwd e{1.5}1\kern-1pt}
\def\veStar{{\mathop{\kern0pt e\lower1.5pt\hbox{${}^*$}}\limits^{\kern0pt
   \raise.02ex\vbox to 0pt{\hbox{$\scriptscriptstyle\rightarrow$}\vss}}}}
\def\eStarv{{\mathop{\kern0pt e\lower1.5pt\hbox{${}^*$}}\limits^{\kern0pt
   \raise.02ex\vbox to 0pt{\hbox{$\scriptscriptstyle\leftarrow$}\vss}}}}
\def\vedash{{\mathop{\kern0pt e\lower.5pt\hbox{${}
     \scriptstyle'$}}\limits^{\kern0pt\raise.02ex
     \vbox to 0pt{\hbox{$\scriptscriptstyle\rightarrow$}\vss}}}}

\def\vE{\kern-.7pt\fwd E3\kern-.7pt}

\def\vA{\kern-.7pt\fwd A3\kern-.7pt}

\def\vEE{\kern-.7pt\fwd{\E}3\kern-.7pt}
\def\vF{\kern-.7pt\fwd F3\kern-.7pt}

\def\vG{\kern-.7pt\fwd G3\kern-.7pt}
\def\vH{\kern-.5pt\fwd H3\kern-.5pt}
\def\vP{\kern-.7pt\fwd P3\kern-.6pt}

\def\specrel#1#2{\mathrel{\mathop{\kern0pt #1}\limits_{#2}}}

\title{The fundamental group of a locally finite\\ graph with ends}

\author{Reinhard Diestel and Philipp Spr\"ussel}
 \date{}

\begin{document}      

\maketitle

\begin{abstract}
We characterize the fundamental group of a locally finite graph $G$ with ends combinatorially, as a group of infinite words. Our characterization gives rise to a canonical embedding of this group in the inverse limit of the free groups $\pi_1(G')$ with $G'\sub G$ finite.
   \end{abstract}

\section{Introduction}

The purpose of this paper is to give a combinatorial characterization of the fundamental group of the compact space $|G|$ formed by a locally finite graph~$G$---such as a Cayley graph of a finitely generated group---together with its ends. The space~$|G|$, known as the \emph{Freudenthal compactification\/} of~$G$, is the standard setting in which locally finite graphs are studied from a topological point of view~\cite{DiestelBook05}. However, no combinatorial characterization of its fundamental group has so far been known.

When $G$ is finite, $\pi_1(|G|) = \pi_1(G)$ is the free group on the set of (arbitrarily oriented) \emph{chords} of a spanning tree of~$G$, those edges of $G$ that are not edges of the tree. When $G$ is infinite and there are infinitely many chords, then $\pi_1(|G|)$ is not a free group. However, we show that it embeds canonically as a subgroup in an inverse limit $F^*$ of free groups: those on the finite sets of (oriented) chords of any \emph{topological} spanning tree~$T$, one whose closure in $|G|$ contains no non-trivial loop.

More precisely, we characterize $\pi_1(|G|)$ in terms of subgroup embeddings
 $$\pi_1(|G|)\to F_\infty \to F^*,$$
where $F_\infty$ is a group formed by `reduced' infinite words of chords of~$T$. These words arise as the traces of loops in~$|G|$, so in general they will have arbitrary countable order types. Unlike for finite graphs, many natural homotopies between such loops do not proceed by retracting passes through chords one by one. (We give a simple example in Section~\ref{wordsec}.) Nevertheless, we show that to generate the homotopy classes of loops in $|G|$ from suitable representatives we only need homotopies that do retract passes through chords one at a time, in some linear order. As a consequence, we are again able to define reduction of words as a linear sequence of steps each cancelling one pair of letters, although the order in which the steps are performed may now have any countable order type (such as that of the rationals).

The fact that our sequences of reduction steps are not well-ordered will make it difficult or impossible to handle reductions in terms of their definition. However we show that reduction of infinite words can be characterized in terms of the reductions they induce on all their finite subwords. A formalization of this observation yields the embedding $F_\infty \to F^*$.

An end of $G$ is \emph{trivial} if it has a contractible neighbourhood. If every end of $G$ is trivial, then $|G|$ is homotopy equivalent to a finite graph. If $G$ has exactly one non-trivial end, then $|G|$ is homotopy equivalent to the Hawaiian Earring. Its fundamental group was studied by Higman~\cite{HigmanFreeProd} and Cannon~\& Conner~\cite{CannonConnerER}. Our characterization of $\pi_1(|G|)$ is equivalent to their combinatorial description of this group when $G$ has only one non-trivial end.

Our motivation for this paper is primarily that that the fundamental group of such a classical space as $|G|$ ought to be understood. Our characterization achieves this aim while remaining as close to the standard representation of $\pi_1(G)$ for finite $G$ as possible; indeed we shall see that the only added complication, the fact that both words and reductions may now have arbitrary (countable) order type, is necessary. 

Our characterization of $\pi_1(|G|)$ already has a first substantial application. In~\cite{Hom1} we use it to show that, in contrast to finite graphs, the first singular homology of~$|G|$ differs essentially from the \emph{topological cycle space} $\C(G)$ of~$G$. It has been amply demonstrated in recent years---see e.g.\ \cite{LocFinTutte,Duality,Partition,LocFinMacLane,AgelosFleisch,Arboricity}, or \cite{RDsBanffSurvey} for a survey---that the relatively new notion of~$\C(G)$, rather than the usual finitary cycle space, is needed to describe the homology of a locally finite graph. But it had remained an open problem whether $\C(G)$ was a truly new object, or just the first singular homology group of~$|G|$ in a new guise. This question was answered positively in~\cite{Hom1}: $\C(G)$~differs essentially from~$H_1(|G|)$, and Theorem~\ref{pi1thm} below is the cornerstone of the proof.

This paper is organized as follows. We begin with a section collecting together the definitions and known background that we need; some elementary general lemmas are also included here. In Section~\ref{wordsec} we introduce our group $F_\infty$ of infinite words, and show how it embeds in the inverse limit of the free groups on its finite subsets of letters. In Section~\ref{sec:pi1toFinf} we embed $\pi_1(|G|)$ in~$F_\infty$, leaving the proof of the main lemma to Section~\ref{sec:injectivity}.

\section{Terminology and basic facts}

In this section we briefly run through any non-standard terminology we use. We also list a few elementary lemmas that we shall need, and use freely, later on.
Some of these are given with references,
the others are proved for the sake of completeness. The reader is encouraged to skim this section for definitions, but to turn to the proofs of the lemmas only as needed.

For graphs we use the terminology of~\cite{DiestelBook05}, for topology that of Hatcher~\cite{Hatcher}. Our graphs may have multiple edges but no loops. This said, we shall from now on use the terms \emph{path} and \emph{loop} topologically, for continuous but not necessarily injective maps $\sigma\colon [0,1]\to X$, where $X$ is any topological space. If $\sigma$ is a loop, it is \emph{based at} the point~$\sigma(0) = \sigma(1)$. We write $\sigma^-$ for the path $s\mapsto \sigma(1-s)$. The image of an injective path is an \emph{arc} in~$X$, the image of an `injective loop' (a subspace of $X$ homeomorphic to~$S^1$) is a \emph{circle} in~$X$.

\begin{lemma}[\cite{ElemTop}]\label{arc}
The image of a topological path with distinct endpoints $x,y$ in a Hausdorff space $X$ contains an arc in $X$ between $x$ and~$y$.
\end{lemma}

All homotopies between paths that we consider are relative to the first and last point of their domain, usually~$\{0,1\}$. We shall often construct homotopies between paths segment by segment. The following lemma enables us to combine certain homotopies defined separately on infinitely many segments.

\begin{lemma}
   \label{infhomotopies}
  Let $\alpha,\beta$ be paths in a topological space~$X$. Assume that there is a sequence $(a_0,b_0),(a_1,b_1),\dotsc$ of disjoint subintervals of $[0,1]$ such that $\alpha$ and $\beta$ conincide on $[0,1]\sm\bigcup_n(a_n,b_n)$, while each segment $\alpha \restr [a_n,b_n]$ is homotopic in $\alpha([a_n,b_n])\cup\beta([a_n,b_n])$ to $\beta \restr [a_n,b_n]$. Then $\alpha$ and $\beta$ are homotopic.
\end{lemma}

\begin{proof}
  Write $D\assign\bigcup_n(a_n,b_n)$. For every $n\in\N$ let $F^n={(f^n_t)}_{t\in[0,1]}$ be a homotopy in $\alpha([a_n,b_n])\cup\beta([a_n,b_n])$ between $\alpha \restr [a_n,b_n]$ and $\beta \restr [a_n,b_n]$. We define the desired homotopy $F={(f_t)}_{t\in[0,1]}$ between $\alpha$ and $\beta$ as
  \begin{equation*}
    f_t(x)\assign
    \begin{cases}
      f^n_t(x) & \text{if }x\in(a_n,b_n),\\
      \alpha(x)=\beta(x) & \text{if }x\in[0,1]\sm D.
    \end{cases}
  \end{equation*}
  Clearly, $f_0=\alpha$ and $f_1=\beta$. It remains to prove that $F$ is continuous.

Let $x,t\in[0,1]$ and a neighbourhood $U$ of $F(x,t)$ in $X$ be given. We find an $\eps>0$ so that $F((x-\eps,x],(t-\eps,t+\eps)) \subset U$; the case $F([x,x+\eps),(t-\eps,t+\eps)) \subset U$ is analogous. Suppose first that there is an $\eps_0>0$ such that $(x-\eps_0,x)\subset D$. As the intervals $(a_i,b_i)$ are disjoint, this means that $(x-\eps,x)\subset(a_n,b_n)$ for some~$n$. Then $(x-\eps_0,x]\subset[a_n,b_n]$, and hence $F \restr (x-\eps_0,x]\times[0,1] = F^n \restr (x-\eps_0,x]\times[0,1]$. As $F^n$ is continuous, there is an $\eps<\eps_0$ with $F((x-\eps,x],(t-\eps,t+\eps)) \subset U$.

Now suppose that for every $\eps > 0$ the interval $(x-\eps,x)$ meets $[0,1]\sm D$. Then also $x\in [0,1]\sm D$, and hence $F(x,t)=\alpha(x)=\beta(x)$. Pick $\eps>0$ with $x-\eps\in[0,1]\sm D$ small enough that both $\alpha$ and $\beta$ map $[x-\eps,x]$ into $U$. Then $F((x-\eps,x],(t-\eps,x+\eps))\subset U$. Indeed, for every $x'\in(x-\eps,x]\sm D$ and every $t'\in(t-\eps,t+\eps)$ we have $F(x',t')=\alpha(x')=\beta(x')\in U$. On the other hand, for every $x'\in(x-\eps,x]\cap D$ and $t'\in(t-\eps,t+\eps)$ we have $x'\in(a_n,b_n)$ for some $n$. As $x$ and $x-\eps$ lie in $[0,1]\sm D$, we have $(a_n,b_n)\subset(x-\eps,x)$ and hence $F(x',t')=F^n(x',t')\in\alpha([a_n,b_n])\cup\beta([a_n,b_n])\subset U$.
\end{proof}

Locally finite CW-complexes can be compactified by adding their \emph{ends}. This compactification can be defined, without reference to the complex, for any connected, locally connected, locally compact topological space $X$ with a countable basis. Very briefly, an \emph{end} of $X$ is an equivalence class of sequences $U_1\supseteq U_2\supseteq \ldots$ of connected non-empty open sets with compact boundaries and an empty overall intersection of closures, $\bigcap_n\overline U_n = \emptyset$, where two such sequences $(U_n)$ and $(V_m)$ are \emph{equivalent} if every $U_n$ contains all sufficiently late $V_m$ and vice versa. This end is said to \emph{live in} each of the sets~$U_n$, and every $U_n$ together with all the ends that live in it is \emph{open} in the space whose point set is the union of $X$ with the set $\Omega(X)$ of its ends and whose topology is generated by these open sets and those of~$X$. This is a compact space, the \emph{Freudenthal compactification} of~$X$ \cite{Freudenthal31, Freudenthal42}. More topological background on this can be found in~\cite{AbelsStrantzalos, BauesQuintero, HughesRanicki}; for applications to groups see e.g.~\cite{AbelsStrantzalos, RoggiEndsI, RoggiEndsII, ThomassenWoess, woessBook}.

For graphs, ends and the Freudenthal compactification are more usually defined combinatorially~\cite{DiestelBook05, halin64, jung71}, as follows. Let $G$ be a connected locally finite graph. A 1-way infinite graph-theoretical path in $G$ is a \emph{ray}. Two rays are \emph{equivalent} if no finite set of vertices separates them in~$G$, and the resulting equivalence classes are the \emph{ends} of~$G$. It is not hard to see~\cite{Ends} that this combinatorial definition of an end coincides with the topological one given earlier for locally finite complexes. We write $\Omega = \Omega(G)$ for the set of ends of~$G$. The Freudenthal compactification of~$G$ is now denoted by~$|G|$; its topology is generated by the open sets of $G$ itself (as a 1-complex) and the sets $\hat C (S,\omega)$ defined for every end $\omega$ and every finite set $S$ of vertices, as follows. $C(S,\omega) =: C$ is the unique component of $G-S$ in which $\omega$ \emph{lives} (i.e., in which every ray of~$\omega$ has a \emph{tail}, or subray), and $\hat C (S,\omega)$ is the union of $C$ with the set of all the ends of $G$ that live in $C$ and the (finitely many) open edges between $S$ and~$C$.%
   \footnote{The definition given in~\cite{DiestelBook05} is slightly different, but equivalent to the simpler definition given here when $G$ is locally finite. Generalizations are studied in~\cite{KroenEnds, ThomassenVellaContinua}.}
Note that the boundary of $\hat C (S,\omega)$ in $|G|$ is a subset of~$S$, that every ray converges to the end containing it, and that the set of ends is totally disconnected.

Many topological spaces that are not normally associated with graphs can be expressed as a graph with ends, or as a subspace thereof. The Hawaiian Earring, for example, is homeomorphic to the subspace of the infinite grid that consists of all the vertical double rays and its end. Since the subspaces of graphs with ends form a richer class than the spaces of graphs with ends themselves, we prove all our results not just for  $|G|$ but more generally for subspaces $H$ of~$|G|$. However, the reader will lose little by thinking of $H$ as the entire space~$|G|$. The subspaces we shall be considering will be \emph{standard} subspaces of~$|G|$: connected subspaces that are closed in $|G|$ and contain every edge of which they contain an inner point.

We shall frequently use the following non-trivial lemma.

\begin{lemma}[\cite{TST}]\label{arccntd}
  For a locally finite graph $G$, every closed, connected subspace of $|G|$ is arc-connected.
\end{lemma}

A~\emph{topological tree in~$|G|$} is an arc-connected standard subspace of~$|G|$ that contains no circle. Note that the subgraph that such a space induces in $G = |G|\sm\Omega$ need not be connected: its arc-connectedness may hinge on the ends it contains. A~\emph{chord} of a topological tree $T$ is any edge of $G$ that has both its endvertices in $T$ but does not itself lie in~$T$.

\begin{lemma}
   \label{locarccon}
   Topological trees in $|G|$ are locally arc-connected.
\end{lemma}

\begin{proof}
Let $T$ be a topological tree in~$|G|$. Let $D$ be any open subset of~$T$, and $x\in D$. We have to find an arc-connected open neighbourhood of $x$ in $T$ inside~$D$. This is trivial if $x$ is a vertex or an inner point of an edge, so we assume that $x$ is an end. Then $D$ may be chosen of the form $D = \hat C(S,x)\cap T$, for some finite set $S\sub V(G)$. Since $G-S$ has only finitely many components, $T\setminus S$ is a finite union of open sets of this form, so $D$ is open and closed in~$T\setminus S$.

Similarly, $T\setminus S$ has only finitely many arc-components, and hence only finitely many components. Each of them is closed and open in~$T\setminus S$, and open even in~$T$. One of them, $C_x$~say, contains~$x$. Then $C_x\sub D$, since $D$ is open and closed in $T\sm S$. To complete the proof, we show that $C_x$ is arc-connected. 

Suppose not. As $C_x$ is the union of some of the finitely many arc-components of $T\setminus S$, it has only finitely many arc-components. Not all of them can be closed in $C_x$, since $C_x$ is connected. Let $C$ be an arc-component of $C_x$ that is not closed in~$C_x$. Then its closure $\Cbar$ in~$T$ meets~$C_x\sm C$, and clearly $\Cbar\cap (C_x\sm C)\sub\Omega$.

Since the components of $T\sm S$ other than $C_x$ are open in~$T$, we have $\Cbar\sub C_x\cup S$.
   As~$C$ is connected and $T$ is closed in~$|G|$, we know that $\Cbar$~is connected and closed in~$|G|$, and hence arc-connected by Lemma~\ref{arccntd}. Let $A$ be an arc in $\Cbar$ from a point in $C$ to one in~$C_x\sm C$. As $S$ is finite and $\Cbar\sm(S\cup C)\sub\Omega$ contains no arc, we can choose $A$ so that $A\cap S=\es$. But then $A\sub C_x$, contradicting the definition of $C$ as an arc-component of~$C_x$.
\end{proof}

Between any two of its points, $x$ and~$y$ say, a topological tree $T$ in~$|G|$ contains a unique arc, which we denote by~$xTy$. These arcs are `short' also in terms of the topology that $|G|$ induces on~$T$:

\begin{lemma}\label{shortarcs}
  If a sequence $z_0,z_1,\dotsc$ of points in $T$ converges to a point~$z$, then every neighbourhood of~$z$ contains all but finitely many of the arcs~$z_iTz_{i+1}$.
\end{lemma}

\begin{proof} Since the arcs~$z_i T z_{i+1}$ are unique, Lemma~\ref{locarccon} implies that they lie in arbitrarily small neighbourhoods of~$z$.
\end{proof}

We shall need topological trees in $|G|$ as spanning trees for our analysis of~$\pi_1(|G|)$: arbitrary graph-theoretical spanning trees of $G$ can have non-trivial loops in their closures, which would leave no trace of chords and thus be invisible to our intended representation of homotopy classes by words of such chords.

Let us call a topological tree $T$ in $|G|$ a \emph{\tst} of~$G$ if $T$ contains~$V(G)$. Since $T$ is closed in~$|G|$, it then also contains~$\Omega(G)$. Similarly, a topological tree $T$ in~$|G|$ is a \emph{\tst} of a subspace $H$ of $|G|$ if $T\sub H$ and $T$ contains every vertex or end of $G$ that lies in~$H$.

Topological spanning trees are known to exist in all locally finite connected graphs (and in many more~\cite{DiestelBook05, DiestelLeaderBGC, DiestelLeaderNST}). They also exist in all the relevant subspaces. We need a slight technical strengthening of this:

\begin{lemma}
   \label{Hspanning}
  Let $T\sub H$ be standard subspaces of~$|G|$. If $T$ is a topological tree, it can be extended to a topological spanning tree of~$H$.
\end{lemma}

\begin{proof}
  As $G$ is locally finite and connected, $|G|$~is a compact Hausdorff space~\cite{DiestelBook05}. Let $\S$ be the set of standard subspaces of~$|G|$ such that $T\sub S\sub H$ and $S$ contains all the vertices and ends of $G$ that lie in~$H$. Every $S\in\S$ is closed in~$|G|$, and therefore compact. Since the intersection of a nested chain of compact connected Hausdorff spaces is connected~\cite[p.~203]{Willard}, $\S$~has a minimal element~$T'$ by Zorn's Lemma. By Lemma~\ref{arccntd}, $T'$~is arc-connected, and it contains no circle: if it did, we could delete an edge to obtain a smaller element of~$\S$. (Since $V(G)\cup\Omega(G)$ is totally disconnected, every circle in $|G|$ contains an edge.) Hence $T'$ is a topological tree in~$|G|$, and by definition of~$\S$ a topological spanning tree of~$H$ containing~$T$.
\end{proof}

Like graph-theoretical trees, topological trees in $|G|$ are contractible. We shall need a slightly technical strengthening of this. Call a homotopy $F(x,t)$ \emph{time-injective} if for every $x$ the map $t\mapsto F(x,t)$ is either constant or injective.

\begin{lemma}
   \label{TST}
For every point $x$ in a topological tree~$T$ in~$|G|$ there is a time-injective deformation retraction of $T$ onto~$x$.
\end{lemma}

\begin{proof}
  The space $T$ is metrizable as follows.
   Choose an enumeration of the edges in $T$ and give the $n$th edge length~$2^{-n}$. Define the distance $d(y,z)$ between points $y,z$ in $T$ as the sum of lengths of the edges (and partial edges) in $yTz$; note that if $y\not=z$ then $yTz$ meets the interior of at least one edge. Then clearly $d$ is a metric with $d(y,z)\le 1$ for all $y,z\in T$, and using Lemma~\ref{shortarcs} it is easy to check that it induces the given topology on $T$. Further, if $z\in yTy'$ for some $y,y'\in T$ we have $d(y,y')=d(y,z)+d(z,y')$. We construct a time-injective homotopy $F$ in $T$ from the identity on $T$ to the map $T\to \{x\}$; then we have $F(y,t) \in xTy \sub X$ for every $y\in X$ and $t\in[0,1]$, and hence $F \restr (X\times[0,1])$ will be the desired time-injective homotopy for $X$. For every $y\in T$ and $t\in[0,1]$ let $F(y,t)$ be the unique point on $xTy$ at distance $(1-t)\cdot d(x,y)$ from $x$.

   For the proof that $F$ is continuous, we show that $d(F(y,t),F(y',t)) \le d(y,y')$ for every $y,y'\in T$ and $t\in[0,1]$; then for every $\eps>0$ and every $y,y'\in T$ with $d(y,y')<\eps/2$ and $t,t'\in[0,1]$ with $|t-t'|<\eps/2$ we have
   \begin{align*}
     d(F(y,t),F(y',t')) &\le d(F(y,t),F(y',t)) + d(F(y',t),F(y',t'))\\
     &\le d(y,y') + |t-t'|\cdot d(x,y')\\
     &< \eps/2 + (\eps/2) \cdot 1 = \eps.
   \end{align*}
   As $xTy$ and $xTy'$ are closed, there is a last point $z$ on $xTy$ that also lies in $xTy'$; this point satisfies $xTz = xTy \cap xTy'$ as the unique $x$--$z$~arc $xTz$ is contained in both $xTy$ and $xTy'$. Then $yTz \cup zTy'$ is a $y$--$z$~arc in $T$ and hence $yTy'=yTz \cup zTy'$. This implies $d(y,y')=d(y,z)+d(z,y')$. If $F(y,t)\in zTy$ and $F(y',t)\in zTy'$, then
   \begin{equation*}
      d(F(y,t),F(y',t)) \le d(F(y,t),z) + d(z,F(y',t)) \le d(y,z)+d(z,y') = d(y,y').
   \end{equation*}
   Otherwise at least one of $F(y,t),F(y',t)$ lies in $xTz = xTy \cap xTy'$ and hence both $F(y,t)$ and $F(y',t)$ are contained in $xTy$ or in $xTy'$. In particular, one of $F(y,t),F(y',t)$ lies on the arc between the other and~$x$. Then
   \begin{align*}
      d(F(y,t),F(y',t)) &= |d(x,F(y,t))-d(x,F(y',t))|\\
      &= (1-t)\cdot |d(x,y)-d(x,y')| \le d(y,y').
   \end{align*}\vskip-\baselineskip
\end{proof}

Given a standard subspace $H$ of~$|G|$, let us call an end~$\omega$ of $G$ \emph{trivial in~$H$} if $\omega\in H$ and $\omega$ has a contractible neighbourhood in~$H$. For instance, all the ends of $G$ are trivial in every \tst\ of~$G$, by Lemma~\ref{TST}. Trivial ends in larger subspaces can also be made visible by \tst s:

\begin{lemma}
   \label{trivial}
  Let $T$ be a topological spanning tree of a standard subspace $H$ of $|G|$. An end $\omega\in H$ of $G$ is trivial in $H$ if and only if $\omega$ has a neighbourhood in $H$ that contains no chord of~$T$.
\end{lemma}

\begin{proof}
  Suppose first that $\omega$ has a neighbourhood in~$H$ containing no chord of~$T$. This neighbourhood $U$ can be chosen of the form $\hat C(S,\omega)\cap H$, since these form a neighbourhood basis of~$\omega$, and so that the $S$--$C$ edges in $H$ are no chords of $T$ either. Then~$U$, indeed its closure $\Ubar$ in~$H$, contains no inner point of any chord of~$T$, i.e., $\Ubar\sub T$. By Lemma~\ref{locarccon}, there is an arc-connected neighbourhood $U'\sub U$ of $\omega$ in~$H$, which we may clearly choose so that its closure $T'$ in $H$ is a standard subspace of~$|G|$. Then $T'\sub\Ubar\sub T$, and $T'$ is arc-connected by Lemma~\ref{arccntd}. So $T'$ is a topological tree in~$|G|$, and contractible by Lemma~\ref{TST}.

  Conversely, suppose that $\omega$ has a contractible neighbourhood $U$ in~$H$; this cannot contain a circle. By Lemma~\ref{TST}, the end $\omega$~has an open arc-connected neighbourhood $T'$ in~$T$ inside~$U$. Since $T$ carries the subspace topology from~$H$, this has the form $T'=U'\cap T$ for an open subset $U'\sub U$ of~$H$. This $U'$ is a neighbourhood of~$\omega$ in~$H$ that contains no chord of~$T$: for any such chord it would also contain an arc in~$T'\sub U$ between its vertices, to form a circle in $U$ that does not exist.
\end{proof}

An edge $e=uv$ of $G$ has two {\em directions\/}, $(u,v)$ and~$(v,u)$. A~triple $(e,u,v)$ consisting of an edge together with one of its two directions is an {\em oriented edge\/}. The two oriented edges corresponding to $e$ are its two {\em orientations\/}, denoted by $\ve$ and~$\ev$. Thus, $\{\ve,\ev\} = \{(e,u,v), (e,v,u)\}$, but we cannot generally say which is which. However, from the definition of $G$ as a CW-complex we have a fixed homeomorphism $\theta_e\colon [0,1]\to e$. We call $(\theta_e(0),\theta_e(1))$ the {\em natural direction\/} of~$e$, and $(e,\theta_e(0),\theta_e(1))$ its {\em natural orientation\/}.

Let $\sigma\colon [0,1]\to |G|$ be a path in~$|G|$. Given an edge $e = uv$ of~$G$, if $[s,t]$ is a subinterval of $[0,1]$ such that $\{\sigma(s),\sigma(t)\} = \{u,v\}$ and $\sigma((s,t)) = \interior e$, we say that $\sigma$ \emph{traverses~$e$} on~$[s,t]$. It does so \emph{in the direction of~$(\sigma(s), \sigma(t))$}, or \emph{traverses $\ve = (e,\sigma(s),\sigma(t))$}. We then call its restriction to $[s,t]$ a \emph{pass of $\sigma$ through~$e$}, or~$\ve$, \emph{from $\sigma(s)$ to~$\sigma(t)$}.

Using that $[0,1]$ is compact and $|G|$ is Hausdorff, one easily shows that a path in $|G|$ contains at most finitely many passes through any given edge:

\begin{lemma}
   \label{pass}
A path in $|G|$ traverses each edge only finitely often.
\end{lemma}

\begin{proof}
  Let $\sigma$ be a path in $|G|$, and let $e=uv$ be an edge such that $\sigma$ contains infinitely many passes $\sigma\restr [s_n,t_n]$ through~$e$, $n=1,2,\dots$. Passing to a subsequence if necessary, we may assume that the sequence $s_1,s_2,\dots$ converges, say to $s\in [0,1]$. Then the sequence of the corresponding $t_n$ also converges to~$s$: given $\epsilon > 0$, choose $m$ large enough that for all $n > m$ both $|s_n - s| < \epsilon/2$ and $t_n - s_n < \epsilon/2$ (using that the lengths of the intervals $[s_n,t_n]$ converge to~0, which they clearly do); then $|t_n - s| < \epsilon$ for all $n > m$. But now $\sigma$ fails to be continuous at~$s$, because $\{\sigma(s_n), \sigma(t_n)\} = \{u,v\}$ for each~$n$ but each $u,v$ has a neighbourhood not containing the other.
\end{proof}

\section{Infinite words, and limits of free groups}\label{wordsec}

In the this section and the next, we give a combinatorial description of~$\pi_1(|G|)$---indeed of $\pi_1(H)$ for any standard subspace $H$ of~$|G|$, when $G$ is any connected locally finite graph. Our description will involve infinite words and their reductions in a continuous setting, and embedding the group they form as a subgroup of a limit of finitely generated free groups. Such things have been studied also by Eda~\cite{EdaInfFreeProd}, Cannon~\& Conner~\cite{CannonConnerER}, and Chiswell~\& M\"uller~\cite{ChiswellMueller}.

When $G$ is finite, $\pi_1(|G|)$~is the free group $F$ on the set of \emph{chords} (arbitrarily oriented) of any fixed spanning tree, the edges of $G$ that are not edges of the tree. The standard description of $F$ is given in terms of reduced words of those oriented chords, where reduction is performed by cancelling adjacent inverse pairs of letters such as $\ve_i\ev_i$ or~$\ev_i\ve_i$. The map assigning to a path in $|G|$ the sequence of chords it traverses defines the canonical group isomorphism between $\pi_1(|G|)$ and~$F$; in particular, reducing the words obtained from homotopic paths yields the same reduced word.

Our description of $\pi_1(|G|)$ when $G$ is infinite will be similar in spirit, but more complex. We shall start not with an arbitrary spanning tree but with a topological spanning tree of~$|G|$. Then every path in $|G|$ defines as its `trace' an infinite word in the oriented chords of that tree, as before. However, these words can have any countable order type, and it is no longer clear how to define the reduction of words in a way that captures homotopy of paths.

Consider the following example. Let $G$ be the infinite ladder, with a topological spanning tree~$T$ consisting of one side of the ladder, all its rungs, and its unique end~$\omega$ (Figure~\ref{fig:singleladder}). The path running along the bottom side of the ladder and back is a null-homotopic loop. Since it traces the chords $\ve_0, \ve_1,\dotsc$ all the way to~$\omega$ and then returns the same way, the infinite word $\ve_0\ve_1\dotso\ev_1\ev_0$ should reduce to the empty word. But it contains no cancelling pair of letters, such as $\ve_i\ev_i$ or~$\ev_i\ve_i$.

\begin{figure}[htbp]
\centering
\includegraphics[width=.7\linewidth]{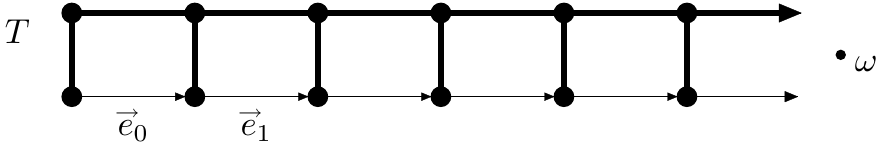}
\caption{The null-homotopic loop $\ve_0 \ve_1\dots\omega\dots\ev_1\ev_0$} 
\label{fig:singleladder}
\end{figure}

This simple example suggests that some transfinite equivalent of cancelling pairs of letters, such as cancelling inverse pairs of infinite sequences of letters, might lead to a suitable notion of reduction. However, in graphs with infinitely many ends one can have null-homotopic loops whose trace of chords contains no cancelling pair of subsequences whatsoever: 

\begin{example}\label{T2}
  There is a locally finite graph $G$ with a null-homotopic loop $\sigma$ in~$|G|$ whose trace of chords contains no cancelling pair of subsequences, of any order type.
\end{example}

\begin{proof}
Let $T$ be the binary tree with root $r$. Write $V_n$ for the set of vertices at distance~$n$ in~$T$ from~$r$, and let $T_n$ be the subtree of $T$ induced by ${V_0\cup\dots\cup V_n}$. Our first aim will be to construct a loop $\sigma$ in $|T|$ that traverses every edge of $T$ once in each direction. We shall obtain $\sigma$ as a limit of similar loops $\sigma_n$ in $T_n\sub |T|$.

Let $\sigma_0$ be the unique (constant) map $[0,1]\to T_0$. Assume inductively that $\sigma_n\colon [0,1]\to T_n$ is a loop traversing every edge of $T_n$ exactly once in each direction. Assume further that $\sigma_n$ pauses every time it visits a vertex in $V_n$ (i.e.,\ a~leaf of $T_n$), remaining stationary at that vertex for some time. More precisely, we assume for every vertex $v\in V_n$ that $\sigma_n^{-1}(v)$ is a non-trivial closed interval. Let us call the restriction of $\sigma_n$ to such an interval a {\em pass\/} of $\sigma_n$ through~$v$.

Let $\sigma_{n+1}$ be obtained from $\sigma_n$ by replacing, for each vertex $v$ in~$V_n$, the pass of $\sigma_n$ through $v$ by a topological path that first travels from $v$ to its first neighbour in~$V_{n+1}$ and back, and then to its other neighbour in~$V_{n+1}$ and back, pausing at each of those neighbours for some non-trivial time interval. Outside the passes of $\sigma_n$ through leaves of~$T_n$, let $\sigma_{n+1}$ agree with~$\sigma_n$.

Let us now define~$\sigma$. Let $s\in [0,1]$ be given. If its values $\sigma_n (s)$ coincide for all large enough~$n$, let $\sigma(s) := \sigma_n (s)$ for these~$n$. If not, then $s_n := \sigma_n (s)\in V_n$ for every~$n$, and $s_0 s_1 s_2\dots$ is a ray in~$T$; let $\sigma$ map $s$ to the end of $G$ containing that ray. This map $\sigma$ is easily seen to be continuous, and by Lemma~\ref{TST} it is null-homotopic. It is also easy to check that no sequence of passes of $\sigma$ through the edges of $T$ is followed immediately by the inverse of this sequence.

   The edges of $T$ are not chords of a topological spanning tree, but this can be achieved by changing the graph: just double every edge and subdivide the new edges once. The new edges together with all vertices and ends then form a topological spanning tree in the resulting graph~$G$, whose chords are the original edges of our tree~$T$, and $\sigma$ is still a (null-homotopic) loop in~$|G|$.
\end{proof}

Example~\ref{T2} shows that there is no hope of capturing homotopies of loops in terms of word reduction defined recursively by cancelling pairs of inverse subwords, finite or infinite. We shall therefore define the reduction of infinite words differently, though only slightly. We shall still cancel inverse letters in pairs, even one at a time, and these reduction `steps' will be ordered linearly (rather unlike the simultaneous dissolution of all the chords by the homotopy in the example). However, the reduction steps will not be well-ordered.

This definition of reduction is less straightforward, but it has an important property: as for finite~$G$, it will be purely combinatorial in terms of letters, their inverses, and their linear order, making no reference to the interpretation of those letters as chords and their relative positions under the topology of~$|G|$. 

Another problem, however, is more serious: since the reduction steps are not well-ordered, it will be difficult to handle reductions---e.g.\ to prove that every word reduces to a unique reduced word, or that word reduction captures the homotopy of loops, i.e.\ that traces of homotopic loops can always be reduced to the same word. The key to solving these problems will lie in the observation that the property of being reduced can be characterized in terms of all the finite subwords of a given word. We shall formalize this observation by way of an embedding of our group $F_\infty$ of infinite words in the inverse limit $F^*$ of the free groups on the finite subsets of letters.

The remainder of this section is devoted to carrying out this programme. In Sections~\ref{sec:pi1toFinf} and~\ref{sec:injectivity} we shall then study how $\pi_1(|G|)$ embeds as a subgroup in~$F_\infty$ when its letters are interpreted as oriented chords of a \tst\ of~$G$. We shall prove that, as in the finite case, the map assigning to a loop in $|G|$ its trace of chords and reducing that trace is well defined on homotopy classes, giving us injective homomorphisms
 $$\pi_1(|G|)\to F_\infty\to F^*\,.$$
By determining their precise images we shall complete our combinatorial characterization of~$\pi_1(|G|)$---and likewise of $\pi_1(H)$ for subspaces~$H$ of~$|G|$.

\medbreak

Let $\vA = \{\ve_0, \ve_1, \dots\}$ and $\{\ev_0, \ev_1,\dots\}$ be disjoint countable sets. Let us call the elements of
 $$A:= \{\ve_0, \ve_1, \dots\}\cup\{\ev_0, \ev_1,\dots\}$$
\emph{letters}, and say that $\ve_i$ and~$\ev_i$ are \emph{inverse} to each other. A~\emph{word} in~$A$ is a map $w\colon S\to A$ from a totally ordered countable set~$S$, the set of \emph{positions} of (the letters used by)~$w$, such that $w^{-1}(a)$ is finite for every $a\in A$. The only property of $S$ relevant to us is its order type, so two words $w\colon S\to A$ and $w'\colon S'\to A$ will be considered the same if there is an order-preserving bijection $\varphi\colon S\to S'$ such that $w=w'\circ \varphi$. If $S$ is finite, then $w$ is a \emph{finite} word; otherwise it is~\emph{infinite}. The \emph{concatenation} $w_1 w_2$ of two words is defined in the obvious way: we assume that their sets $S_1,S_2$ of positions are disjoint, put $S_1$ before~$S_2$ in~$S_1\cup S_2$, and let $w_1w_2$ be the combined map $w_1\cup w_2$. For $I\sub \N$ we let
 $$A_I\assign\{\ve_i \mid i\in I\}\cup \{\ev_i \mid i\in I\}\,,$$
and write $w\restr I$ as shorthand for the restriction $w\restr w^{-1}(A_I)$. Note that if $I$ is finite then so is the word~$w\restr I$, since $w^{-1}(a)$ is finite for every~$a$.

An \emph{interval} of~$S$ is a~sub\-set $S'\sub S$ closed under betweenness, i.e., such that whenever $s'<s<s''$ with $s',s''\in S'$ then also $s\in S'$. The most frequently used intervals are those of the form $[s',s'']_S\assign\{s\in S \mid {s' \le s \le s''}\}$  and $(s',s'')_S\assign\{s\in S \mid s'<s<s''\}$. If $(s',s'')_S=\es$, we call $s',s''$ \emph{adjacent} in~$S$.

A~\emph{reduction} of a finite or infinite word $w\colon S\to A$ is a totally ordered set $R$ of disjoint 2-element subsets of~$S$ such that the two elements of each $p\in R$ are adjacent in $S\sm\bigcup\{q\in R \mid q<p\}$ and are mapped by $w$ to inverse letters~$\ve_i,\ev_i$. We say that \emph{$w$ reduces to} the word $w \restr (S\sm\bigcup R)$. If $w$ has no nonempty reduction, we call it \emph{reduced}.

Informally, we think of the ordering on~$R$ as expressing time. A~reduction of a finite word thus recursively deletes cancelling pairs of (positions of) inverse letters; this agrees with the usual definition of reduction in free groups. When $w$ is infinite, cancellation no longer happens `recursively in time', because $R$ need not be well ordered.

As is well known, every finite word $w$ reduces to a unique reduced word, which we denote as~$r(w)$. Note that $r(w)$ is unique only as an abstract word, not as a restriction of~$w$: if $w = \ve_0\ev_0\ve_0$ then $r(w) = \ve_0$, but this letter $\ve_0$ may have either the first or the third position in~$w$. The set of reduced finite words forms a group, with multiplication defined as $(w_1,w_2)\mapsto r(w_1 w_2)$, and identity the empty word~$\es$. This is the free group with free generators $\ve_0,\ve_1,\dots$ and inverses $\ev_0,\ev_1\dots$. For finite $I\sub \N$, the subgroup
 $$F_I := \{w\mid \im w\sub A_I\}$$
 is the free group on $\{\ve_i \mid i\in I\}$.

Consider a word~$w$, finite or infinite, and $I\sub \N$. It is easy to check the following:
\begin{equation}\begin{minipage}[c]{0.72\textwidth}\label{inducedreduction}\em
If $R$ is a reduction of~$w$ then $\big\{\{s,s'\}\in R\mid w(s)\in A_I\big \}$, with the ordering induced from~$R$, is a reduction of~$w\restr I$.
  \end{minipage}\ignorespacesafterend\end{equation}

\noindent
   In particular:
\begin{equation}\begin{minipage}[c]{0.72\textwidth}\label{inducedreductioninformal}\em
Any result of first reducing and then restricting a word can also be obtained by first restricting and then reducing it.\looseness=-1
  \end{minipage}\ignorespacesafterend\end{equation}

By~\eqref{inducedreductioninformal}, mapping $w\in F_J$ to $r(w\restr I)\in F_I$ for $I\sub J$ defines an inverse system of homomorphisms $F_J\to F_I$. Let us write
 $$F^* := F^*(\vA) := \varprojlim F_I$$
 for the corresponding inverse limit of the~$F_I$. By our assumption that $I$ runs through all the finite subsets of some countable set, and $F_I$ can be viewed as the free group on~$I$, this defines $F^*$ uniquely as an abstract group.

\medbreak

Our next aim is to show that also every infinite word reduces to a unique reduced word. We shall then be able to extend the map $w\mapsto r(w)$, defined so far only for finite words~$w$, to infinite words~$w$. The operation $(w_1,w_2)\mapsto r(w_1 w_2)$ will then make the set of reduced (finite or infinite) words into a group, our desired group~$F_\infty$.

Existence is immediate:

\begin{lemma}\label{lemma:reduce}
  Every word reduces to some reduced word.
\end{lemma}

\begin{proof}
  Let $w\colon S\to A$ be any word. By Zorn's Lemma there is a maximal reduction $R$ of~$w$. Since $R$ is maximal, the word $w\restr (S\sm\bigcup R)$ is reduced.
\end{proof}

To prove uniqueness, we begin with a characterization of the reduced words in terms of reductions of their finite subwords.
Let $w\colon S\to A$ be any word. If $w$ is finite, call a position $s\in S$ \emph{permanent} in~$w$ if it is not deleted in any reduction, i.e., if $s\in S\sm\bigcup R$ for every reduction $R$ of~$w$. If $w$ is infinite, call a position $s\in S$ \emph{permanent} in $w$ if there exists a finite $I\sub\N$ such that $w(s)\in A_I$ and $s$ is permanent in $w\restr I$. By~\eqref{inducedreductioninformal}, a permanent position of $w\restr I$ is also permanent in $w\restr J$ for all finite $J\supseteq I$. The converse, however, need not hold: it may happen that $\{s,s'\}$ is a pair (`of cancelling positions') in a reduction of~$w\restr I$ but $w\restr J$ has a letter from $A_J\sm A_I$ whose position lies between $s$ and~$s'$, so that $s$ and $s'$ are permanent in~$w\restr J$.

\begin{lemma}\label{permanent}
  A word is reduced if and only if all its positions are permanent.
\end{lemma}

\begin{proof}
  The assertion is clear for finite words, so let $w\colon S\to A$ be an infinite word. Suppose first that all positions of $w$ are permanent. Let $R$ be any reduction of~$w$; we will show that $R=\es$.
Let $s$ be any position of~$w$. As $s$ is permanent, there is a finite $I\sub\N$ such that $w(s)\in A_I$ and $s$ is not deleted in any reduction of~$w\restr I$. By~\eqref{inducedreduction}, the pairs in $R$ whose elements map to~$A_I$ form a reduction of~$w\restr I$, so $s$ does not lie in such a pair. As $s$ was arbitrary, this proves that $R=\es$.

Now suppose that $w$ has a non-permanent position~$s$. We shall construct a non-trivial reduction of~$w$. For all $n\in\N$ put $S_n := \{s\in S \mid w(s)\in A_{\{0,\dots,n\}}\}$; recall that these are finite sets. Write $w_n$ for the finite word $w\restr I$ with $I=\{0,\dots,n\}$, the restriction of $w$ to~$S_n$. For any reduction $R$ of~$w_{n+1}$, the set $R^- := \big\{\{t,t'\}\in R\mid t,t'\in S_n\big\}$ with the induced ordering is a reduction of~$w_n$, by~\eqref{inducedreduction}.

Pick $N\in\N$ large enough that $s\in S_N$. Since $s$ is not permanent in~$w$, every $w_n$ with $n\ge N$ has a reduction in which $s$ is deleted. As $w_n$ has only finitely many reductions, K\"onig's infinity lemma~\cite{DiestelBook05} gives us an infinite sequence $R_N, R_{N+1}, \dots$ in which each $R_n$ is a reduction of~$w_n$ deleting~$s$, and $R_n = R_{n+1}^-$ for every~$n$. Inductively, this implies:
  \begin{txteq}\label{pairs}
  For all $m\le n$, we have $R_m = \big\{\{t,t'\}\in R_n\mid t,t'\in S_m\big\}$, and the ordering of $R_m$ on this set agrees with that induced by~$R_n$.
   \end{txteq}
Let $s'\in S$ be such that $\{s,s'\}\in R_n$ for some~$n$; then $\{s,s'\}\in R_n$ for every $n\ge N$, by~\eqref{pairs}.

Our sequence $(R_n)$ divides the positions of~$w$ into two types. Call a position $t$ of $w$ \emph{essential} if there exists an $n\ge N$ such that $t\in S_n$ and $t$ remains undeleted in~$R_n$; otherwise call $t$ \emph{inessential}. Consider the set

$$R\assign\bigcup_{m\ge N}\bigcap_{n\ge m}R_n$$
 of all pairs of positions of $w$ that are eventually in~$R_n$. Let $R$ be endowed with the ordering $p < q$ induced by all the orderings of $R_n$ with $n$ large enough that $p,q\in R_n$; these orderings are compatible by~\eqref{pairs}. Note that $R$ is non-empty, since it contains~$\{s,s'\}$. We shall prove that $R$ is a reduction of~$w$.

We have to show that the elements of each $p\in R$, say $p = \{t_1,t_2\}$ with $t_1 < t_2$, are adjacent in $S\sm\bigcup\{q\in R \mid q<p\}$. Suppose not, and pick $t\in (t_1,t_2)_S\sm\bigcup\{q\in R \mid q<p\}$. If $t$ is essential, then $t$ is a position of $w_n$ remaining undeleted in~$R_n$ for all large enough~$n$. But then $\{t_1,t_2\}\notin R_n$ for all these~$n$, contradicting the fact that $\{t_1,t_2\}\in R$. Hence $t$ is inessential. Then $t$ is deleted in every $R_n$ with $n$ large enough. By~\eqref{pairs}, the pair $\{t,t'\}\in R_n$ deleting~$t$ is the same for all these~$n$, so $\{t,t'\} =: p'\in R$. By the choice of~$t$, this implies $p'\not< p$. For $n$ large enough that $p,p'\in R_n$, this contradicts the fact that $t_1,t_2$ are adjacent in $S_n\sm\bigcup\{q\in R_n, q<p\}$, which they are since $R_n$ is a reduction of~$w_n$.
\end{proof}

Note that a word can consist entirely of non-permanent positions and still reduce to a non-empty word: the word $\ve_0\ev_0\ve_0$ is again an example.

Lemma~\ref{permanent} offers an easy way to check whether an infinite word is reduced. In general, it can be hard to prove that a given word $w$ has no non-trivial reduction, since this need not have a `first' cancellation. By Lemma~\ref{permanent} it suffices to check whether every position becomes permanent in some large enough but finite~$w\restr I$.

Similarly, it can be hard to prove that two words reduce to the same word. The following lemma provides an easier way to do this, in terms of only the finite restrictions of the two words:

\begin{lemma}\label{reduceonI}
  Two words $w,w'$ can be reduced to the same (abstract) word if and only if $r(w\restr I) = r(w'\restr I)$ for every finite $I\sub\N$.
\end{lemma}

\begin{proof}
The forward implication follows easily from~\eqref{inducedreductioninformal}. Conversely, suppose that $r(w\restr I) = r(w'\restr I)$ for every finite $I\sub\N$. By Lemma~\ref{lemma:reduce}, $w$~and $w'$ can be reduced to reduced words $v$ and~$v'$, respectively. Our aim is to show that $v=v'$, that is to say, to find an order-preserving bijection $\phi\colon S\to S'$ between the domains $S$ of $v$ and $S'$ of~$v'$ such that $v = v'\circ\phi$. For every finite~$I$, our assumption and the forward implication of the lemma yield
 $$r(v\restr I)  = r(w\restr I) = r(w'\restr I) = r(v'\restr I)\,.$$
 Hence for every possible domain $S_I\sub S$ of $r(v\restr I)$ and every possible domain $S'_I\sub S'$ of~$r(v'\restr I)$ there exists an order isomorphism $S_I\to S'_I$ that commutes with $v$ and~$v'$. For every~$I$, there are only finitely many such maps $S_I\to S'_I$, since there are only finitely many such sets $S_I$ and~$S'_I$. And for $I\sub J$, every such map $S_J\to S'_J$ induces such a map $S_I\to S'_I$ with $S_I\sub S_J$ and $S'_I\sub S'_J$, by \eqref{inducedreductioninformal}. Hence by the infinity lemma~\cite{DiestelBook05} there exists a sequence $\phi_0\sub \phi_1\sub\dotsb$ of such maps $\phi_n\colon S_{\{0,\dots,n\}}\to S'_{\{0,\dots,n\}}$, whose union~$\phi$ maps all of $S$ onto~$S'$, since by Lemma~\ref{permanent} every position of $v$ and of $v'$ is permanent.
\end{proof}

With Lemma~\ref{reduceonI} we are now able to prove:

\begin{lemma}\label{uniquereduced}
  Every word reduces to a unique reduced word. 
\end{lemma}

\begin{proof}
  By Lemma~\ref{lemma:reduce}, every word $w$ reduces to some reduced word~$w'$. Suppose there is another reduced word $w''$ to which $w$ can be reduced. By the easy direction of Lemma~\ref{reduceonI}, we have
 $$r(w'\restr I) = r(w\restr I) = r(w''\restr I)$$
 for every finite~$I\sub\N$. By the non-trivial direction of Lemma~\ref{reduceonI}, this implies that $w'$ and~$w''$ can be reduced to the same word. Since $w'$ reduces only to~$w'$ and $w''$ reduces only to~$w''$, this must be the word $w'=w''$.
\end{proof}

As in the case of finite words, we denote the unique reduced word that a word $w$ reduces to by~$r(w)$. The set of reduced words now forms a group
 $$F_\infty = F_\infty(\vA)\,,$$
with multiplication defined as $(w_1,w_2)\mapsto r(w_1w_2)$, identity the empty word~$\es$, and inverses $w^-$ of $w\colon S\to A$ defined as the map on the same~$S$, but with the inverse ordering, satisfying $\{w(s),w^-(s)\} = \{\ve_i,\ev_i\}$ for some~$i$ for every~${s\in S}$. (Thus, $w^-$~is $w$ taken backwards, replacing each letter with its inverse.) Note that the proof of associativity requires an application of Lemma~\ref{uniquereduced}.

As indicated earlier, we claim that $F_\infty$ embeds canonically in the inverse limit $F^*$ of the groups $F_I$. By~\eqref{inducedreductioninformal}, the maps $h_I\colon w\mapsto r(w\restr I)$ are homomorphisms $F_\infty\to F_I$ that commute with the homomorphisms $F_J\to F_I$ from the inverse system, so they define a homomorphism
   $$h\colon F_\infty\to F^*$$
satisfying $\pi_I\circ h = h_I$ for all~$I$ (where $\pi_I$ is the projection $F^* \to F_I$). To show that $h$ is injective, consider an element $w$ of its kernel. For every~$I$, we have
 $$r(w\restr I) = h_I (w) = \pi_I(h(w)) = \pi_I({\rm id}_{}) = \es,$$
where $\rm id$ denotes the identity in $F^*$ and $\es$ that of~$F_I$, the empty word. Thus, $w$~is a reduced word which has no permanent positions. By Lemma~\ref{permanent}, this means that $w=0$. Thus, $h$~is a group embedding of~$F_\infty$ in~$F^*$, as claimed.

We remark that $h$ is never surjective. Indeed, while every letter occurs only finitely often in a given word, there are elements of $F^*$ whose projections to the $F_I$ contain some fixed letter unboundedly often; such an element will not lie in the image of~$h$. (For example, the words $\ve_1\ve_0\ve_1\ev_0\    \dots\  \ve_i\ve_0\ve_i\ev_0 \in F_I$ for $I = \{1,\dots,i\}$ define such an element of~$F^*$.) However, these are clearly the only elements of $F^*$ that $h$ misses: the subgroup $h(F_\infty)$ of $F^*$ consists of precisely those elements $(w_I)$ of $F^*$ that are \emph{bounded} in the sense that for every letter $\ve\in A$ there exists a $k\in\N$ such that $|w_I^{-1}(\ve)|\le k$ for all~$I$.

Theorem~\ref{pi1thm}$\,$\eqref{FinfProjLim} below summarizes what we have shown so far.

\section{\boldmath Embedding $\pi_1(H)$ in~$F_{\infty}$}\label{sec:pi1toFinf}

Let $G$ be a locally finite connected graph. Let $H$ be a standard subspace of~$|G|$, and let $T$ be a fixed topological spanning tree of~$H$. If $T$ has only finitely many chords in~$H$, then $H$ is homotopy equivalent to a finite graph---apply Lemma~\ref{TST} to the maximal topological subtrees not meeting the interior of an arc between two chords---and all we shall prove below will be known. We therefore assume that $T$ has infinitely many chords in~$H$. Enumerate these as $e_0,e_1,\dotsc$, let $\vA := \{\ve_0, \ve_1, \dots\}$ be the set of their natural orientations, and put\looseness=-1
 $$A:= \{\ve_0, \ve_1, \dots\}\cup\{\ev_0, \ev_1,\dots\}\,.$$
 Let
 $$F_\infty = F_\infty(\vA)$$
 be the group of infinite reduced words with letters in~$A$, as defined in Section~\ref{wordsec}.

Unless otherwise mentioned, the endpoints of all paths considered from now on will be vertices or ends, and any homotopies between paths will be relative to~$\{0,1\}$. When we speak of `the passes' of a given path~$\sigma$, without referring to any particular edges, we shall mean the passes of $\sigma$ through chords of~$T$.

Every path $\sigma$ in $H$ defines a word~$w_{\sigma}$ by its passes through the chords of~$T$. Formally, we take as $S$ the set of the domains $[a,b]$ of passes of~$\sigma$, ordered naturally as internally disjoint subsets of~$[0,1]$, and let $w_{\sigma}$ map every $[a,b]\in S$ to the directed chord that $\sigma$ traverses on~$[a,b]$. We call $w_{\sigma}$ the \emph{trace} of $\sigma$. Our aim is to show that $\langle\alpha\rangle\mapsto r(w_\alpha)$ defines a group embedding $\pi_1(H)\to F_\infty$.

For a proof that $\langle\alpha\rangle\mapsto r(w_\alpha)$ is well defined, consider homotopic loops $\alpha\sim\beta$ in~$H$. We wish to show that $r(w_\alpha) = r(w_\beta)$. By Lemma~\ref{reduceonI} it suffices to show that $r(w_{\alpha}\restr I) = r(w_{\beta}\restr I)$ for every finite $I\subset\N$. Consider the space obtained from $H$ by attaching a 2-cell to $H$ for every $j\notin I$, by an injective attachment map from the boundary of the 2-cell onto the fundamental circle of~$e_j$, the unique circle in~$T+e_j$. This space deformation-retracts onto $T\cup\bigcup\{e_i\mid i\in I\}$, and hence is homotopy equivalent by Lemma~\ref{TST} to the wedge sum $W_I$ of $|I|$ circles, one for every~$e_i$. Composing $\alpha$ and $\beta$ with the map $H\to W_I$ from this homotopy equivalence yields homotopic loops $\alpha'$ and $\beta'$ in~$W_I$, whose traces in $F_I$ are $w_{\alpha'} = w_{\alpha}\restr I$ and $w_{\beta'} = w_{\beta}\restr I$. Since $\langle\gamma\rangle\mapsto r(w_\gamma)$ is known to be well defined for wedge sums of finitely many circles, we deduce that
$$r(w_{\alpha}\restr I) = r(w_{\alpha'}) = r(w_{\beta'}) = r(w_{\beta}\restr I)\,.$$
  This completes the proof that $\langle\alpha\rangle\mapsto r(w_\alpha)$ is well defined, and by the uniqueness of reduction it is a homomorphism. For injectivity, we shall prove in Section~\ref{sec:injectivity} the following extension to paths that need not be loops:

\begin{lemma}\label{lemma:reducible}
  Paths $\sigma,\tau$ in $H$ with the same endpoints are homotopic in $H$ if (and only if) their traces reduce to the same word.
\end{lemma}

We remark that the map $\langle\alpha\rangle\mapsto r(w_\alpha)$ will not normally be surjective. For example, a sequence $\ve_0,\ve_1,\dots$ of distinct chords will always be a reduced word, but no loop in $|G|$ can pass through these chords in order if they do not converge to an end. Hence if two ends are non-trivial in~$H$, then by Lemma~\ref{trivial} there is a non-convergent sequence $\ve_0,\ve_1,\dots$ of chords of $T$ in~$H$ (picked alternately from smaller and smaller neighbourhoods of the two ends), which forms a reduced word in $F_\infty(\vA)$ outside the image of our map $\langle\alpha\rangle\mapsto r(w_\alpha)$.

In order to describe the image of this map precisely, let us call a subword $w':= w\restr S'$ of a word $w\colon S\to A$ \emph{monotonic} if $S'$ is infinite and can be written as $S' = \{s_0,s_1,\dots\}$ so that either $s_0 < s_1 < \dots$ or ${s_0 > s_1 > \dots}$~. Let us say that $w'$ \emph{converges} (in~$|G|$) if there exists an end to which every sequence $x_0, x_1,\dots$ with $x_n\in w(s_n)$ for all~$n$ converges. If $w$ is the trace of a path in~$H$, then by the continuity of this path all the monotonic subwords of~$w$---and hence those of~$r(w)$---converge.

We can now summarize our combinatorial description of~$\pi_1(H)$ as follows.

\begin{theorem}\label{pi1thm}
   Let $G$ be a locally finite connected graph, and let $H$ be a standard subspace of~$|G|$. Let $T$ be a topological spanning tree of~$H$, and let $e_0, e_1,\dots$ be its chords in~$H$.
\begin{enumerate}[\rm (i)]
  \item \label{pi1Finf}
    The map $\langle\alpha\rangle\mapsto r(w_\alpha)$ is an injective homomorphism from $\pi_1(H)$ to the group $F_\infty$ of reduced finite or infinite words in $\{\ve_0, \ve_1, \dots\}\cup\{\ev_0, \ev_1,\dots\}$. Its image consists of those reduced words whose monotonic subwords all converge in~$|G|$.
  \item \label{FinfProjLim}
    The homomorphisms $w\mapsto r(w\restr I)$ from $F_\infty$ to~$F_I$ embed $F_\infty$ as a subgroup in~$\varprojlim F_I$. It consists of those elements of~$\varprojlim F_I$ whose projections $r(w\restr I)$ use each letter only boundedly often. (The bound may depend on the letter.)
\end{enumerate}
\end{theorem}

\begin{proof}
  \eqref{pi1Finf} We already saw that $\langle\alpha\rangle\mapsto r(w_\alpha)$ is a homomorphism, and injectivity follows from Lemma~\ref{lemma:reducible} (which will be proved in Section~\ref{sec:injectivity}). We have also seen that for every loop~$\alpha$ in~$H$ all the monotonic subwords of $r(w_\alpha)$ converge in~$|G|$. It remains to show the converse: that if all the monotonic subwords of a reduced word $w$ converge, then there is a loop $\alpha$ in $H$ such that $w=r(w_{\alpha})$.

  We prove the following more general fact: If $w$ is a word (not necessarily reduced) whose monotonic subwords all converge, then $w$ is the trace of a loop in~$H$. So let $w\colon S\to A$ be such a word. Enumerate~$S$ as $s_0, s_1,\dots$. We will inductively choose disjoint closed intervals $I_n\sub [0,1]$ ordered correspondingly, i.e.\ so that $I_m$ precedes $I_n$ in $[0,1]$ whenever $s_m < s_n$. For each~$n$, we will let $\alpha_n$ be an order-preserving homeomorphism from $I_n$ to the oriented chord~$w(s_n)$. We will then extend the union of all the $\alpha_n$ to a loop $\alpha\colon [0,1]\to H$.

In order that such a continuous extension $\alpha$ exist, we have to take some precautions when we choose the~$I_n$. For example, suppose that the chords $w(s_0),w(s_2),\dotsc$ converge to one end, while the chords $w(s_1),w(s_3),\dotsc$ converge to another end. If $S$ is ordered as $s_0 < s_2 < \dotsb \vert \dotsb < s_3 < s_1$ (note that every monotonic subword of $w$ converges), there may be a point $x\in [0,1]$ such that every interval around $x$ contains all but finitely many of the intervals~$I_n$. In this case, any extension of $\bigcup_n \alpha_n$ will fail to be continuous at~$x$. In order to prevent this, we shall first formalize such critical situations in terms of partitions of~$S$, then prove that there are only countably many of them, reserve open intervals as padding around potentially critical points such as~$x$, and finally choose the $I_n$ so as to avoid these intervals.

Consider a partition $P$ of $S$ into non-empty parts $S^-(P)$ and~$S^+(P)$, such that $s^-<s^+$ whenever $s^-\in S^-(P)$ and $s^+\in S^+(P)$. Note that for all co\-final sequences in $S^-(P)$, finite or infinite, the final vertices of the corresponding chords of $T$ converge to a common point $z^-(P)\in H$: otherwise there would be a monotonic subword of $w$ that does not converge in~$|G|$. Likewise, there is a point $z^+(P)\in H$ such that for all coinitial sequences in $S^+(P)$ the first vertices of the corresponding chords converge to~$z^+(P)$. We call $P$ \emph{critical} if $z^-(P)\not= z^+(P)$.

Let us show that there are only countably many critical partitions~$P$. If not, there are uncountably many for which neither $S^-(P)$ has a greatest element nor $S^+(P)$ has a least element. For each such $P$, both $z^-(P)$ and $z^+(P)$ are ends. Let $W(P)$ be a finite set of vertices that separates $z^-(P)$ from $z^+(P)$ in~$|G|$. As $G$ is countable, there is a set $W$ such that $W=W(P)$ for uncountably many~$P$. Pick a sequence $P_0,P_1,\dotsc$ of critical partitions with $W(P_i)=W$, and so that either $S^-(P_0)\subsetneq S^-(P_1)\subsetneq \dots$ or $S^+(P_0)\subsetneq S^+(P_1)\subsetneq \dots$. We assume that $S^-(P_0)\subsetneq S^-(P_1)\subsetneq \dots$, the other case being analogous. To obtain a contradiction, let us use this sequence to construct a non-convergent monotonic subword of~$w$.

  Choose $s^-_0\in S^-(P_0)$ and $s^+_0\in S^+(P_0)\cap S^-(P_1)$ so that $W$ separates $w(s^-_0)$ from $w(s^+_0)$ in~$G$; this is possible, since $W$ separates $z^-(P_0)$ from $z^+(P_0)$. Then for $i=1,2,\dotsc$ in turn choose $s^-_i\in S^-(P_i)$ and $s^+_i\in S^+(P_i)\cap S^-(P_{i+1})$ so that $s^-_i>s^+_{i-1}$ and $W$ separates $w(s^-_i)$ from $w(s^+_i)$ in~$G$. Then $w\restr \{s^-_0,s^+_0,s^-_1,s^+_1,\dotsc\}$ is a monotonic subword of $w$ that does not converge in~$|G|$, since $W$ separates all the pairs $w(s^-_i),w(s^+_i)$. This completes the proof that there are only countably many critical partitions; enumerate them as $P_0,P_1,\dotsc$.

We now construct~$\alpha$. Inductively choose disjoint, closed, non-trivial intervals $I_n,J_n\sub[0,1]$ so that $I_m$ precedes $I_n$ on $[0,1]$ whenever $s_m<s_n$, and so that $I_m$ precedes $J_n$ if and only if $s_m\in S^-(P_n)$. For each~$n$, let $\alpha_n$ be an order-preserving homeomorphism from $I_n$ to the oriented chord~$w(s_n)$. Extend the union of all the~$\alpha_n$ to a loop~$\alpha$, as follows. Consider the connected components $I$ of $[0,1]\sm \bigcup I_n$. These are again intervals, possibly trivial; we call them \emph{connecting intervals}. We shall first define $\alpha$ on the boundary points $a\le b$ of each~$I$, and then extend it continuously to map $\overline I$ onto the arc $\alpha(a)T\alpha(b)$. To help with our subsequent proof that $\alpha$ is continuous, let us make sure when we choose $\alpha(a)$ and $\alpha(b)$ that they satisfy the following for $x=a$ or $x=b$:
\begin{txteq}\label{precontinuous}
  For every boundary point $x$ of a connecting interval, and every neighbourhood $U$ of $\alpha(x)$, there is an $\eps>0$ such that $\alpha(y)\in U$ whenever $y\in(x-\eps,x+\eps)$ lies in an interval $I_n$.
\end{txteq}

Suppose first that $I$ is not an initial or final segment of $[0,1]$, i.e.,\ that the boundary points $a,b$ of $I$ satisfy $0<a\le b<1$. Now the sets
 $$S^-\assign\{s_n \mid I_n\text{ precedes }I\}\text{ and }
   S^+\assign\{s_n \mid I\text{ precedes }I_n\}$$
  form a partition $P$ of $S$ into non-empty parts. If $P$ is critical, then $P=P_m$ for some $m$ and hence $J_m\sub I$ by the choice of $I$ as a component of $[0,1]\sm\bigcup I_n$. In particular, $I$ is non-trivial in this case. Let $\alpha$ map $a$ to $z^-(P)$ and $b$ to $z^+(P)$. (Note that if one of these points, $a$~say, is a boundary point of some $I_n$, then $\alpha_n(a)=z^-(P)$, so $\alpha$ coincides with $\alpha_n$ on $I_n$.) If $P$ is not critical, we have $z^-(P)=z^+(P)$ and let $\alpha$ map $a$ and $b$ to this point. In both these cases, \eqref{precontinuous} follows easily from the definition of $z^-(P)$ and~$z^+(P)$.

  Now consider the case $a=0$. If $S$ has a least element~$s_n$, then $b$ is a boundary point of $I_n$ on which $\alpha$ is already defined. Otherwise, $S$~has a coinitial sequence $\dotsc < s_{n_1} < s_{n_0}$, in which case the restriction of $w$ to this sequence is a monotonic subword of~$w$; this converges in $|G|$ to some end~$\omega$, and we put $\alpha(b) =\omega$. Note that $\omega$ is independent of the choice of the coinitial sequence of~$S$, as otherwise there would be a monotonic subword of $w$ that does not converge. In both cases, our choice of $\alpha(b)$ satisfies~\eqref{precontinuous}. As $\alpha(a)$, if $a\not=b$, we choose any vertex or end in~$H$, satisfying~\eqref{precontinuous} trivially. In the case of $b=1$, we proceed analogously.

  For each~$I$, we now extend $\alpha$ to $\overline I\to \alpha(a)T\alpha(b)$ as planned. In particular, if $\alpha(a)=\alpha(b)$ we let $\alpha$ map all of $I$ to that point.

It remains to check that $\alpha$ is continuous at boundary points $x$ of intervals $I_n$ or of connecting intervals. Suppose first that $x$ is the boundary point of a connecting interval. Let $x_0,x_1,\dots$ be a sequence of points in $[0,1]$ converging to~$x$. If all but finitely many of these lie in intervals~$I_n$, their images under $\alpha$ converge to $\alpha(x)$ by~\eqref{precontinuous}. If not, we may assume that each $x_i$ lies in some connecting interval~$I^i = [y_i,z_i]$. At most one of these intervals~$I$ contains infinitely many~$x_j$ (because then it contains~$x$), whose values then converge to~$\alpha(x)$ by the continuity of~$\alpha\restr I$. Disregarding these $x_j$, we may thus assume that each $I^i$ contains only finitely many~$x_j$.

Let us show that the sequence $\alpha(y_0), \alpha(z_0), \alpha(y_1), \alpha(z_1),\dots$ converges to~$\alpha(x)$.
Let a neighbourhood $U\assign\hat C(S,\alpha(x))$ of $\alpha(x)$ be given, and let $\eps$ be as provided by~\eqref{precontinuous}. For all but finitely many $i$ we have $y_i,z_i\in(x-\eps,x+\eps)$, and we claim that $\alpha(y_i),\alpha(z_i)\in U$ for such $i$. By definition of~$y_i$ (the case of $z_i$ is analogous), there is a sequence of boundary points of intervals $I_n$ that converges to~$y_i$, and we may choose this sequence in~$(x-\eps,x+\eps)$. By~\eqref{precontinuous}, $\alpha$~maps these points to vertices converging to~$\alpha(y_i)$. By our choice of~$\eps$, these vertices lie in~$U$. As every sequence of vertices in $U$ converges to a point in~$U$, we obtain $\alpha(y_i)\in U$ as desired.
We now apply Lemma~\ref{shortarcs} to the sequence $\alpha(y_0), \alpha(z_0), \alpha(y_1), \alpha(z_1),\dotsc$. By the lemma, the entire arcs $\alpha(I^i) = \alpha(y_i) T \alpha(z_i)$ converge to~$\alpha(x)$. In particular, $\alpha(x_i)\to \alpha(x)$ as desired.

Suppose now that $x$ is not the boundary point of a connecting interval but of an interval $I_n$, say $I_n=[x,y]$. Then the sets $S^-\assign\{s_m \mid I_m\text{ precedes }x\}$ and $S^+\assign\{s_m \mid x\text{ precedes }I_m\}$ form a partition $P$ of~$S$, with $s_n = \min S^+$. If $S^-=\es$, then $x=0$ and continuity at $x$ is trivial. Otherwise, $S^-(P)$ has a greatest element or $P$ is critical. In both cases, $x$ would be the boundary point of a connecting interval, a contradiction.

Note that $\alpha$ as defined here does not need to be a loop. But we can turn it into a loop without changing its trace, by appending to it a path in $T$ from $\alpha(1)$ to $\alpha(0)$.

  \eqref{FinfProjLim} This was proved at the end of Section~\ref{wordsec}.
\end{proof}

\begin{corollary}\label{subgroup}
  Let $H\sub H'$ be standard subspaces of~$|G|$. Then $\pi_1(H)$ is a subgroup of $\pi_1(H')$.
\end{corollary}

\begin{proof}
  Let $T$ be a topological spanning tree of~$H$. By Lemma~\ref{Hspanning}, $T$~extends to a \tst~$T'$ of~$H'$. Since every chord of $T$ is also a chord of~$T'$, Theorem~\ref{pi1thm}$\,$\eqref{pi1Finf} implies the assertion.
\end{proof}

Let us call a standard subspace $H$ of~$|G|$ \emph{non-trivial} if it contains an end (of~$G$) that is non-trivial in~$H$. The fundamental groups of all such spaces contain and are contained in the abstract group~$F_\infty$ defined in Section~\ref{wordsec}; recall that this group is independent of~$G$, as long as $|G|$ itself is non-trivial. Indeed:

\begin{corollary}\label{equivalentgroups}
  For every non-trivial stan\-dard subspace $H$ of any~$|G|$ there are subgroup embeddings $F_\infty\to\pi_1(H)\to F_\infty$.
\end{corollary}

\begin{proof}
  Theorem~\ref{pi1thm}$\,$\eqref{pi1Finf} says that $\pi_1(H)$ is a subgroup of~$F_\infty$. Conversely, let $T$ be a \tst\ of~$H$. Since $H$ is non-trivial, $T$~has a sequence of chords in~$H$ that converge to an end $\omega\in H$ of~$G$ (Lemma~\ref{trivial}). The union~$H'$ of~$T$ with all these chords is a standard subspace of $|G|$ contained in~$H$, so $\pi_1(H')\le\pi_1(H)$ by Corollary~\ref{subgroup}. Since $T$ is a \tst\ of $H'$ all whose (infinite subwords of words of) chords converge in~$|G|$, Theorem~\ref{pi1thm}$\,$\eqref{pi1Finf} implies that $\pi_1(H')$ is isomorphic to~$F_\infty$.
\end{proof}

\begin{corollary}\label{freetrivial}
The fundamental group $\pi_1(H)$ of a standard subspace $H$ of~$|G|$ is free if and only if $H$ is trivial. In particular, $\pi_1(|G|)$~is free if and only if every end has a contractible neighbourhood in~$|G|$.
\end{corollary}

\begin{proof}
  Let $T$ be a \tst\ of~$H$. If $H$ is trivial, then $T$~has only finitely many chords in~$H$: otherwise some of them would converge to a end, which would be non-trivial in $H$ by Lemma~\ref{trivial}. By Lemma~\ref{TST}, $H$~is homotopy equivalent to a finite graph whose fundamental group is the free group on this set of chords.

Conversely, let us assume that $H$ is non-trivial and show that $\pi_1(H)$ is not free. By Corollary~\ref{equivalentgroups}, $\pi_1(H)$~contains $F_\infty$ as a subgroup, so by the Nielsen-Schreier theorem~\cite{Stillwell80} it suffices to prove that $F_\infty$ is not free. As pointed out in~\cite{CannonConnerER}, this was shown by Higman~\cite{HigmanFreeProd}.
\end{proof}

When $H$ is non-trivial, then $\pi_1(H)$ is uncountable, so any representation needs uncountably many generators. We believe that one also needs uncountably many relations, but have no proof of this.

Theorem~\ref{pi1thm} provides a reasonably complete solution to our original graph-theoretical problem, which asked for a canonical combinatorial description of the fundamental group of~$|G|$ for given~$G$. It does not answer the group-theoretical question of how interesting or varied are the groups occurring as~$\pi_1(|G|)$ or~$\pi_1(H)$.

We close with some evidence that this question may indeed be interesting. For all we know so far, $F_\infty$~might be the only abstract group ever occurring as $\pi_1(H)$ for non-trivial~$H$. However, this is far from the truth:

\begin{theorem}\label{oneend}
  The fundamental group $\pi_1(H)$ of a standard subspace $H$ of $|G|$ is isomorphic to $F_\infty$ if and only if $H$ contains precisely one end of $G$ that is non-trivial in $H$.
\end{theorem}

For the proof of Theorem~\ref{oneend} we need a lemma
about~$F_\infty$. By Theorem~4.1 of Conner and Eda \cite{ConnerEda},
the image of any homomorphism $F_\infty\to A*B$ is contained in either
$A'*B$ or $A*B'$ for finitely generated subgroups $A'$ of~$A$ and $B'$
of~$B$. Hence,

\begin{lemma}\label{quasiatomic}
  If groups $A,B$ are not finitely generated then there is no
  epimorphism $F_\infty \to A*B$.\noproof
\end{lemma}

\begin{proof}[Proof of Theorem~\ref{oneend}]
  Suppose first that $H$ contains precisely one end of $G$ that is
  non-trivial in~$H$. Then every sequence of chords in~$H$ in which
  each chord appears only finitely often converges to this end. The group embedding $\pi_1(H)\to F_\infty$ from
  Theorem~\ref{pi1thm}$\,$\eqref{pi1Finf} is therefore surjective.

  Now suppose that $\pi_1(H)$ is isomorphic to $F_\infty$. Then $H$ is
  non-trivial (see the proof of Corollary~\ref{freetrivial}), so $H$ contains an end of $G$ that is non-trivial in~$H$. It remains
  to show that there cannot be two such ends, $\omega_1$
  and~$\omega_2$ say. Let $T$ be a topological spanning tree of~$H$,
  and pick a finite set $S$ of vertices separating $\omega_1$
  from~$\omega_2$. Consider the graph $G'=G/S$ obtained from $G$ by
  deleting any edges between vertices in~$S$ and identifying all the
  vertices in~$S$ to a new vertex~$v_S$. Since rays in $G$ have a tail
  in $G'$ and vice versa, there is an obvious bicontinuous bijection between the ends of $G$
  and those of~$G'$, and we shall not distinguish them notationally.
  Let $H/S$ and $T/S$ be the quotient spaces of $H$ and~$T$ defined
  analogously.

  Our next aim is to construct a standard subspace $H'$ of $|G'|$ and
  a topological spanning tree $T'$ of $H'$ such that the chords of
  $T'$ are precisely the chords of $T$ other than those between
  vertices in~$S$. Clearly $T/S$ is a path-connected subspace of $H/S$
  that contains all its vertices and ends, but $T/S$ can contain
  circles. However, all these contain~$v_S$, so by deleting some edges
  at $v_S$ we can make $T/S$ into a topological spanning tree $T'$
  of~$H/S$. As we do not want $T'$ to have chords in $H/S$ that are
  not chords of~$T$, we remove the same edges at $v_S$ also
  from~$H/S$, to obtain a standard subspace $H'$ of $|G'|$ of which $T'$ is still
  a topological spanning tree.

  It is easy to see that the chords of $T'$ are precisely the chords
  of $T$ that do not join two vertices in~$S$. We enumerate the chords
  of $T$ as $e_0,e_1,\dotsc$ so that the chords of $T'$ are precisely
  $e_n,e_{n+1},\dotsc$ for some $n\in\N$.

  The vertex $v_S$ separates the ends $\omega_1$ and $\omega_2$
  in~$G'$, so $|G'|\sm\{v_S\}$ is the disjoint union of open sets
  $O_1,O_2$ with $\omega_1\in O_1$ and $\omega_2\in O_2$. For $i=1,2$
  let $H_i\assign\{v_S\}\cup(H'\cap O_i)$. Both $H_i$ are non-trivial,
  and hence their fundamental groups are not finitely generated. As $H_1\cap H_2=\{v_S\}$, the Seifert~- van~Kampen theorem yields
  that $\pi_1(H')\simeq\pi_1(H_1)*\pi_1(H_2)$.

  We now define an epimorphism $f\colon \pi_1(H)\to\pi_1(H')$; as
  $\pi_1(H)\simeq F_\infty$ by assumption, this will induce an
  epimorphism $F_\infty\to\pi_1(H_1)*\pi_1(H_2)$ contradicting
  Lemma~\ref{quasiatomic}. By Theorem~\ref{pi1thm}$\,$\eqref{pi1Finf},
  every element $a$ of $\pi_1(H)$ corresponds to a reduced word $w_a$
  in $F_\infty(\{\ve[0],\ve[1],\dotsc\})$ all of whose monotonic
  subwords converge in $|G|$. Consider the word $r(w_a\restr\{n,n+1,
  \dotsc\})$. This word corresponds to an element of~$\pi_1(H')$: its
  monotonic subwords are subwords of~$w_a$, so they converge in~$|G|$
  and hence also in~$|G'|$. Hence by \eqref{inducedreductioninformal} and Theorem~\ref{pi1thm}$\,$\ref{pi1Finf} the map $w_a\mapsto
  r(w_a\restr\{n,n+1,\dotsc\})$ induces a homomorphism $f:\pi_1(H)\to
  \pi_1(H')$. For the same reason $f$ is surjective: by
  Theorem~\ref{pi1thm}$\,$\ref{pi1Finf}, every element of $\pi_1(H')$
  corresponds to a reduced word $w$ in$\{\ve[n],\ve[n+1],\dotsc\}\cup
  \{\ev[n],\ev[n+1],\dotsc\}$ all of whose monotonic subwords converge
  in~$|G'|$, and hence also in~$|G|$. Therefore $w=w_a$ for some
  $a\in\pi_1(H)$, by Theorem~\ref{pi1thm}$\,$\eqref{pi1Finf}.
\end{proof}

\section{Proof of Lemma~\ref{lemma:reducible}}\label{sec:injectivity}

We conclude our proof of Theorem~\ref{pi1thm} with the proof of Lemma~\ref{lemma:reducible}. In this proof we shall need another lemma:

\begin{lemma}\label{reducedpath}
  Let $\sigma$ be a path in $H$ and let $w_{\sigma}:S\to A$ be its trace. Then for every $S'\sub S$ there is a path $\tau$ in $H$ with the same endpoints as $\sigma$ and such that $w_{\tau}=w_{\sigma}\restr S'$. Moreover, $\tau$ can be chosen so that $\tau\restr[a,b]=\sigma\restr[a,b]$ for every domain $[a,b]\in S'$ of a pass of $\tau$.
\end{lemma}

\begin{proof}
  Note that the first statement follows from the more general fact that we proved in Theorem~\ref{pi1thm}(\ref{pi1Finf}): As $w_{\sigma}$ is the trace of a path, all its monotonic subwords converge in $H$. Hence also all monotonic subwords of $w_{\sigma}\restr S'$ converge in $H$, and thus it is the trace of a path~$\tau$. To ensure that $\tau$ coincides with $\sigma$ on every domain of a pass of $\tau$, we will construct $\tau$ explicitely.

  For every $x\in[0,1]$ that lies in an interval in $S'$ we define $\tau(x) \assign \sigma(x)$. Further, we put $\tau(0)\assign\sigma(0)$ and $\tau(1)\assign\sigma(1)$. Then for every $x\in[0,1]$ with $\tau(x)$ still undefined there is a unique maximal interval $[a,b]$ that contains $x$ and is disjoint from~$(s,t)$ for every $[s,t]\in S'$. It is easy to see that $\sigma(a)$ and $\sigma(b)$ are vertices or ends, and hence lie in~$T$. If $a\ne b$, we call $[a,b]$ a \emph{non-traversing interval} and define $\tau$ on $[a,b]$ as a path from $\sigma(a)$ to $\sigma(b)$ whose image is precisely the arc in $T$ between these two points. If $a=b$, then $\sigma(a)$ is an end, and we let $\tau(a) := \sigma(a)$. Clearly $w_{\tau}=w_{\sigma} \restr S'$ and $\tau\restr [a,b]=\sigma \restr [a,b]$ for each $[a,b]\in S'$; it remains to show that $\tau$ continuous.

Continuity is clear at inner points of intervals in $S'$ or of non-traversing intervals, so let $x$ be any other point. By definition, $\tau(x)=\sigma(x)$. It is easy to see that $\sigma(x)$ is a vertex or an end, so $\tau(x)\in T$. Continuity of $\tau$ at $x$ is now easy to prove, using Lemma~\ref{shortarcs} and the fact that $\sigma$ is continuous.
\end{proof}

\begin{proof}[Proof of Lemma~\ref{lemma:reducible}]
The proof that $\alpha\sim\beta$ implies $r(w_{\alpha}) = r(w_{\beta})$ was shown already for the case that $\alpha$ and~$\beta$ are loops. The general case follows, since $\sigma$ and $\tau$ can be made into loops by appending a path in $T$ joining their endpoints, which does not change their traces. It remains to prove the converse: we assume that $r(w_{\sigma}) = r(w_{\tau})$, and show that $\sigma\sim\tau$.

Our aim is to construct a homotopy $F = (f_t)_{t\in[0,1]}$ of paths $f_t$ in~$H$ with $f_0 = \sigma$ and $f_1 = \tau$. We first assume that $\tau$ does not traverse any chords; the general statement will then follow from this case. Our proof of this case will consist of the following five parts. We begin with some simplifications of the problem, straightening $\sigma$ and~$\tau$ to homotopic but less complicated paths. We then pair up the passes of $\sigma$ through chords, with a view to cancel such pairs $(\ve,\ev)$ by a local homotopy $f_t\to f_{t'}$ that retracts a small segment of $f_t$ running through~$\ve$ and back through~$\ev$ without traversing any other chords. These pairs of passes have to be nested in the right way, and finding this pairing will be the second part of our proof. In the third part we determine a (time) order in which to cancel those nested pairs. This order will have to list inner pairs before outer pairs, but among all the linear orders doing this we have to find a suitable one: since the partial order of the nestings of pairs can have limits, so will the linear order of times $t$ at which to cancel the pairs. At these limits we may encounter discontinuities in our homotopy. But we shall be able to choose the order of cancellations so that this happens only countably often. In the fourth step we smooth out those discontinuities by inserting further local homotopies $f_t\to f_{t'}$ at those countably many times~$t$. In part five of the proof we finally show that the homotopy $F$ thus defined is indeed continuous.

We now embark on part one of the proof, simplifying $\sigma$ and~$\tau$. Although $\tau$ does not, by assumption, traverse chords, its image might still contain inner points of chords. Let us call a path $\alpha$ in $H$ \emph{straight} if it does not have this property, i.e.\ if $\alpha(x)\in T$ for every $x$ not contained in the domain of a pass. Applying Lemma~\ref{infhomotopies} to local homotopies retracting any segments of $\alpha$ visiting a chord $e=uv$ to a constant map with image $u$ or~$v$, we can make any path in $H$ straight without touching its passes:
\begin{txteq}\label{straight}
Every path $\alpha$ in $H$ is homotopic in~$\im\alpha$ to a straight path~$\alpha'$ that has the same passes as~$\alpha$.
\end{txteq}
By `having the same passes' we mean not only that $\alpha,\alpha'$ have the same trace but that for every interval $[a,b]\sub [0,1]$ either both segments $\alpha\restr[a,b]$ and $\alpha'\restr[a,b]$ are a pass of their respective path or neither is, and if both are then $\alpha\restr[a,b] = \alpha'\restr[a,b]$.

By~\eqref{straight}, $\sigma$ and $\tau$ are homotopic to straight paths $\sigma'$ and $\tau'$ such that $\sigma'$ has the same passes as $\sigma$ while $\tau'$ has the same passes as $\tau$ (namely, none). In particular, $w_{\sigma'}=w_{\sigma}$ and $w_{\tau'}=w_{\tau} =\es$, so $r(w_{\sigma'}) = r(w_{\sigma}) = r(w_{\tau}) = r(w_{\tau'})$ by assumption. We may therefore assume that $\sigma$ and~$\tau$ are themselves straight.

Let us start with the simplest case: assume that $\sigma$ also traverses no chord. Then $\sigma$ and $\tau$ are homotopic. Indeed:
\begin{txteq}\label{pathsinT}
  Let $\alpha$ and $\beta$ be paths in $T$ with identical endpoints $x,y\in V(G)\cup\Omega(G)$. Then there is a homotopy between $\alpha$ and $\beta$ in $\im\alpha\cup\im\beta$. If $\beta$ is constant, this homotopy can be chosen time-injective.
\end{txteq}

\noindent
To prove~\eqref{pathsinT}, we construct homotopies of $\alpha$ and~$\beta$ to an $x$--$y$ path in~$xTy$. We shall assemble these two homotopies from homotopies between the segments $\gamma\colon [a,b]\to T$ of $\alpha$ or~$\beta$ that are maximal with $\gamma((a,b))\sub T\setminus xTy$, and constant maps, defined as follows. Since $T\sm xTy$ is open in~$T$, the maximality of $[a,b]$ implies that $\gamma(a), \gamma(b)\in xTy$. Let us show that $\gamma(a) = \gamma(b)$. If not, then these two points are joined by two arcs that have disjoint interiors: one in~$xTy$, and another in~$\im\gamma$ (Lemma~\ref{arc}). These two arcs would form a circle in~$T$, which does not exist since $T$ is a topological spanning tree of~$H$. By Lemma~\ref{TST}, there is a time-injective homotopy in~$\im\gamma$ from $\gamma$ to the constant map $[a,b]\to \{\gamma(a)\}$ ($=\{\gamma(b)\}$). Combining all these homotopies, one for every~$\gamma$, we obtain time-injective homotopies (in $\im\alpha$ and $\im\beta$) from $\alpha$ and $\beta$ to $x$--$y$ paths in~$xTy$. If $\beta$ is constant (with image $x=y$), the first of these is the desired time-injective homotopy from $\alpha$ to~$\beta$. Otherwise we note that since $xTy \simeq [0,1]$, the latter two paths are homotopic in~$xTy$, and we can combine our three homotopies to the desired homotopy between $\alpha$ and~$\beta$.

Using Lemma~\ref{infhomotopies} to apply~\eqref{pathsinT} to segments between passes, we obtain the following generalization:
\begin{txteq}\label{samepasses}
  If paths $\alpha$ and $\beta$ in $H$ with identical endpoints have precisely the same passes (i.e., they agree on each domain of a pass of $\alpha$ or $\beta$), then there is a homotopy between them in $\im\alpha\cup\im\beta$.
\end{txteq}

Let us assume now that $\sigma$ traverses chords. Then $[0,1]$ is the disjoint union of the following intervals: the interiors of domains of passes, and the components of the rest of $[0,1]$. Every such component is a closed interval. By~\eqref{samepasses}, we may assume that 
\begin{txteq}\label{imageofsigma}
  If $(a,b)\subset[0,1]$ is maximal with the property that it avoids all domains of passes of $\sigma$, then $\sigma$ maps $[a,b]$ onto the arc~$\sigma(a)T\sigma(b)$. In particular, if $\sigma(a)=\sigma(b)$ then $\sigma$ is constant on~$[a,b]$.
\end{txteq}
   This completes part one of the proof.

\medbreak
   
In order to construct our desired homotopy $F$ between $\sigma$ and~$\tau$, we have to use the fact that $w_{\sigma}$ reduces to $w_{\tau} = \es$. So let us consider a reduction $R$ of $w_{\sigma}$ to the empty word. By definition, $R$~is a totally ordered set of disjoint pairs of positions of $w_{\sigma}$ such that the elements of each $p\in R$ are adjacent in $S\sm\bigcup\{q\in R \mid q<p\}$ and are mapped by $w_{\sigma}$ to inverse letters $\ve_i,\ev_i$. The positions of~$w_\sigma$, in this case, are the domains of the passes of~$\sigma$. Let us write
 $$\P := \big\{(\pi,\rho)\mid \exists (s,t)\in R\,\colon (\pi = \sigma\restr s
     \text{ and } \rho = \sigma\restr t)\big\}$$
 for the set of pairs of passes corresponding to~$R$, where the order of $(s,t)$ is that from~$S$ (so the interval $s$ precedes the interval $t$ in~$[0,1]$). Note that $\P$ is countable, because $\sigma$ has only countably many passes. Our homotopy $F$ will remove the passes of $\sigma$ in pairs as specified by~$\P$, one at a time. The order in which this is done will not necessarily be the ordering which $R$ induces on~$\P$, so let us for now think of $\P$ as an unordered set.

Let us fix some more notation for later use.
Let $\alpha$ be a path in~$H$, and let $p = (\pi,\rho)$ be a pair of passes of $\alpha$ through the same chord $e =: e(p)$ but in opposite directions. Let $[a,a^-]$ be the domain of~$\pi$ and $[b^-,b]$ that of~$\rho$, and assume that $a < a^-\le b^- < b$. Then $\pi(a) = \rho(b) =: z$ and $\pi(a^-) = \rho(b^-)  =: z^-$ are the two vertices of~$e$. If $\alpha\restr[a^-,b^-]$ traverses no chord, and if $\beta$ is the path obtained from $\alpha$ by replacing its segment $\alpha\restr[a,b]$ with the constant map $[a,b]\to\{z\}$, we say that $\beta$ is obtained from $\alpha$ by \emph{cancelling the pair $p$ of passes}. Since $a,a^-,b,b^-,z,z^-$ depend only on the pair $p=(\pi,\rho)$ but not on the rest of~$\alpha$, we denote them by $a(p), a^-(p), b(p), b^-(p), z(p)$ and~$z^-(p)$. Thus:
\begin{txteq}\label{imageofabp}
   For all $p\in\P$, we have $\sigma(a(p))=\sigma(b(p))=z(p)$ and $\sigma(a^-(p))=\sigma(b^-(p)) = z^-(p)$. These points are vertices and hence lie in $T$.
\end{txteq}
This completes the second part of our proof.

\medbreak

In the third part, we now wish to determine the order in which our homotopy $F$ has to cancel the pairs in~$\P$. This order will have to satisfy an obvious necessary condition imposed by the relative position of the passes in these pairs. Indeed, our definition of~$\P$ implies that, given two pairs $p,p'\in\P$, the intervals $(a(p),b(p))$ and $(a(p'),b(p'))$ are either disjoint or nested; accordingly, we call $p$ and $p'$ \emph{parallel} or \emph{nested}. If $p$ and $p'$ are nested and $[a(p),b(p)]$ contains $[a(p'),b(p')]$, we say that $p$ \emph{surrounds} $p'$ and write $p\ge p'$. This is clearly a partial ordering. In fact, $\le$~is the inverse of a tree order:
\begin{txteq}\label{nestedpairs}
Whenever a pair is surrounded by two other pairs, these latter pairs are nested.
\end{txteq}
This partial ordering $\le$ on $\P$ will have to be respected by any order in which our homotopy $F$ can cancel the pairs in~$\P$: if $p$ surrounds~$p'$, then $p'$ has to be cancelled before~$p$.

Our next aim, therefore, is to extend $\le$ to a total ordering $\preceq$ on~$\P$. We may think of $\preceq$ informally as the `order in which $F$ will cancel the pairs', but should remember that the order type of $\preceq$ may be arbitrarily complicated for a countable linear ordering. Note that, in order for our desired homotopy $F(x,t)$ to be continuous, it will not be enough to take as $\preceq$ the (arbitrary) linear extension of $\le$ defined by the reduction~$R$.

To define $\preceq$, we start by enumerating the elements of $\P$ arbitrarily. Then, recursively for $i=0,1,\dotsc$, let $M_i$ be a maximal chain in $\P\sm (M_0\cup\ldots\cup M_{i-1})$ containing the first pair from this set in the enumeration of~$\P$. (If there is no pair left we terminate the recursion, so all the chains $M_i$ are non-empty.)
Given two pairs $p,p'$, not necessarily distinct, let $p\preceq p'$ if either
\begin{itemize}
\item
  $p$, $p'$ lie in the same $M_i$ and $p\le p'$; or
\item
  $p\in M_j$ and $p'\in M_i$ with $i<j$.
\end{itemize}
Thus, $\preceq$ puts later chains below earlier chains. When $p\preceq p'$ we say that $p$ \emph{precedes} $p'$, and $p'$ \emph{succeeds}~$p$. By \eqref{nestedpairs} and the maximality of the chains~$M_i$, each of the sets $M^n := \bigcup_{i=0}^n M_i$ is closed upwards in~$\le\,$: if $p'\ge p\in M^n$ then also $p'\in M^n$. In words:
\begin{txteq}\label{lateissmall}
  For all $i<j$, no pair in $M_j$ surrounds a pair in $M_i$.
\end{txteq}
By definition of~$\preceq$, this implies that $\preceq$ is indeed a linear extension of~$\le$, i.e.\ that $p\preceq p'$ whenever $p\le p'$: every pair precedes any pair that surrounds it.

Having partitioned the set $\P$ of passes of $\sigma$ into pairs, and having chosen an order $\preceq$ in which we want our homotopy $F$ to cancel them, we next wish to map our pairs $p$ to time intervals $[s(p),t(p)]\sub [0,1]$ in which $F$ can cancel~$p$. These intervals will reflect~$\preceq$ in that
\begin{equation}\label{crudeorder}
t(p) < s(p')\text{ \em whenever } p\prec p';
\end{equation}
   in particular, they will be disjoint for different~$p$.
While $t$ runs through $[s(p),t(p)]$, the path $f_t$ will change only on~$[a(p), b(p)]$, so as to cancel~$p$.

In order to state precisely what we require, we need another definition.
Let $\alpha$ be a topological path in~$H$, with $\alpha(a)=\alpha(b) =: z$ for some $a < b$, and let $\beta$ be the path obtained from $\alpha$ by replacing $\alpha\restr[a,b]$ with the constant map $[a,b]\to \{z\}$. We say that a homotopy from $\alpha$ to~$\beta$ \emph{retracts $\alpha\restr[a,b]$ to~$z$} if it is relative to $[0,a]\cup [b,1]$, time-injective, and maps $[a,b]\times[0,1]$ to~$\alpha([a,b])$. If $\alpha([a,b])\sub X\sub H$, we may also say that this retraction is performed \emph{in~$X$}.

Our paths~$f_t$ will satisfy the following assertions:
\begin{txteq}
  The passes of $f_t$ are also passes of~$\sigma$; thus, $f_t \restr [c,d] = \sigma \restr [c,d]$ for every $t\in [0,1]$ and every domain $[c,d]$ of a pass of $f_t$.
\end{txteq}
\begin{txteq}
  For every $p\in\P$, the passes of $f_{s(p)}$ are exactly those passes of $\sigma$ that are contained in pairs $p'\succeq p$; in particular, $p$~is a pair of passes of $f_{s(p)}$.
\end{txteq}
\begin{txteq}\label{localhomotopies}
  For every $p\in\P$, the path $f_{t(p)}$ is obtained from $f_{s(p)}$ by cancelling the pair~$p$. This is achieved by a homotopy $(f_t)_{t\in[s(p),t(p)]}$ that first retracts $f_{s(p)} \restr [a^-(p),b^-(p)]$ to~$z^-(p)$ in~$T$ and then retracts the resulting path $f_t \restr [a(p),b(p)]$ to~$z(p)$ in~$e(p)$.
\end{txteq}

\smallskip
The first part of the homotopy in~\eqref{localhomotopies} will be obtained by applying~\eqref{pathsinT}, with $\alpha = f_{s(p)}\restr[a^-(p),b^-(p)]$ and $\beta\colon [a^-(p),b^-(p)]\to\{z^-(p)\}$. This will turn $f_{s(p)}$ into a path $f_t$ mapping $[a(p),b(p)]$ onto the chord~$e(p)$, to which the second part of the homotopy is then applied.

Let us note for later use:
\begin{equation}\label{meetp}
 \text{\em The path $f_{t(p)}$ maps $[a(p),b(p)]$ to the vertex~$z(p)$.}
\end{equation}

With this preview of how the linear order $\preceq$ of cancellations of passes will be implemented by~$F$, we complete the third part of our proof.

\medbreak

In the fourth part, we now turn to the reason why we have not chosen the intervals $[s(p),t(p)]$ explicitly yet. This is because observing the rules just outlined will not suffice to make our homotopy $F$ continuous.

To see this, consider a bipartition $r = (P^-,P^+)$ of $\P$ into non-empty sets $P^-,P^+$ such that $p^-\prec p^+$ whenever $p^-\in P^-$ and $p^+\in P^+$. Given $i\in\N$, let
$$P^+_i := P^+\cap M_i\quad\emtext{and}\quad P^-_i := P^-\cap M_i\,.$$
Since $P^-$ is non-empty, $P^+$ meets only finitely many~$M_i$. Denote the largest~$i$ with $P^+_i\ne\es$ by~$i(r)$, and call it the \emph{index} of~$r$. Then $P^+_{i(r)}$ is an initial segment of~$P^+$. Since the elements of $M_i$ are nested, $\preceq$ coincides on $P^+_{i(r)}$ with~$\le$.

Let us call $r$ \emph{critical} if $P^+$ has no least element (with respect to~$\preceq$). For critical partitions~$r$, we define
$$[a^+,b^+] := \bigcap_{p\in P^+_{i(r)}} [a(p),b(p)]\,.$$
   \vskip-6pt\noindent
Since $P^+_{i(r)}$ is countable, it has an (infinite) coinitial sequence $p^+_0 > p^+_1 > \dots$. Then $\lim a(p^+_n) = a^+$. As $\sigma$ is continuous, the vertices $z(p^+_n) = \sigma(a(p^+_n))$ converge in $H$ to~$\sigma(a^+)$. As only finitely many of the vertices $z(p^+_n)$ can coincide (Lemma~\ref{pass}), $\sigma(a^+)$~must be an end, which we denote by
  $$z(r) \in\Omega(G)\,.$$

Call a point $x\in [a^+,b^+]$ \emph{critical} (with respect to~$r$) if $x\in (a(q),b(q))$ also for some $q\in P^-$. Then
 $$P_x^- := \big\{q\in P^- \mid x \in (a(q),b(q))\big\}\ne\es$$
is a $\le$-chain, which may or may not have a greatest element. Put
 $$(a_x^-,b_x^-) := \bigcup_{q\in P_x^-} \big(a(q),b(q)\big).$$
   \vskip-6pt\noindent
Every $q\in P_x^-$ is nested with every $p\in P^+_{i(r)}$, since $x\in (a(q),b(q))\cap (a(p),b(p))$. As $q\preceq p$, this means that $q\le p$, so
 $$[a_x^-,b_x^-]\sub [a^+,b^+]\,.$$
The reader may wish to construct some examples of such~$\sigma$ at this point.%
\footnote{Here is one example: Suppose that $\omega$ is a non-trivial end of $H$, and let $e'_0,e'_1,\dotsc$ be a sequence of chords of $T$ that converges to $\omega$. Now let $\sigma$ be a loop in $H$ based at a vertex $v$ that first traverses $\ve'_2,\ve'_3,\dotsc$ and reaches $\omega$ at time $1/3$. For $x\in[2/3,1]$, we let $\sigma(x)\assign\sigma(1-x)$. Thus in $[2/3,1]$, $\sigma$ returns from $\omega$ to~$v$, traversing $\dotsc,\ev'_3,\ev'_2$. Between time $1/3$ and $2/3$ it traverses $\ve'_0,\ve'_1,\ev'_1,\ev'_0$. Let the domains of these passes be $[0.35,0.36],[0.4,0.41],[0.5,0.51]$, and $[0.64,0.65]$. In~$\P$, the unique pass through $\ve'_i$ forms a pair with the unique pass through~$\ev'_i$. In the order~$\prec$, the pairs $p_0,p_1$ that consist of passes through $e'_0$ and through~$e'_1$, respectively, are smaller than alle the other pairs. Now consider the partition $r = (P^-,P^+)$ of~$\P$ with $P^- = \{p_0,p_1\}$. Then $[a^+,b^+]=[1/3,2/3]$, and the points $x\in(0.35,0.65)$ are critical with respect to~$r$. For $x\in (0.4,0.51)$ we have $P^-_x = \{p_0, p_1\}$, for the other $x$ in $(0.35,0.65)$ we have $P_x^- = \{p_0\}$, but for all $x\in(0.35,0.65)$ we have $(a^-_x,b^-_x)=(0.35,0.65)$. The points in $(1/3,2/3)$ outside $(0.35,0.65)$ are not critical with respect to~$r$.}

To see why such $r$ and $x$ are `critical' for the construction of our homotopy~$F$, assume now that $F$ has been chosen so as to satisfy \eqref{crudeorder}--\eqref{localhomotopies}, but apart from this arbitrarily.
Since $p^+_0\succ p^+_1\succ\dots$ we then have $t^+_0 > t^+_1 > \dots$, where $t^+_n := t(p^+_n)$ is the time at which $F$ has just cancelled the pair~$p^+_n$. Let
 $$t^+ := {\textstyle \lim_n t^+_n}\,.$$
For every $x\in [a(p^+_n),b(p^+_n)]$, statement~\eqref{meetp} yields $f_{t^+_n}(x) = z(p^+_n) = \sigma(a(p^+_n))$. Every $x\in [a^+,b^+]$ satisfies this for all~$n$, so for such $x$ the continuity of $F$ and~$\sigma$ imply
\begin{equation}\label{A+}
 {\textstyle
 f_{t^+} (x) = \lim_n f_{t^+_n}(x) = \lim_n \sigma(a(p^+_n)) = \sigma (a^+) = z(r)\,.
 }
\end{equation}
Now assume that $x$ is critical. Since $P^-_x$ is a countable $\le$-chain, it contains a (finite or infinite) cofinal sequence $q_0 < q_1 < \dots$. Then $\lim_n a(q_n) = a^-_x$. Let
 $$t^-_x := {\textstyle \lim_n t(q_n)}\,.$$
As $q\prec p$ for all $q\in P^-$ and $p\in P^+$, we have $t^-_x \le t^+$ by~\eqref{crudeorder}. As earlier, the fact that $x\in [a(q_n),b(q_n)]$ for all~$n$ implies
\begin{equation}\label{Ax}
 {\textstyle
  f_{t^-_x} (x) = \lim_n f_{t(q_n)}(x) \specrel={\eqref{meetp}} \lim_n z(q_n)
    = \lim_n \sigma(a(q_n)) = \sigma (a^-_x) =: z_x\,.
 }
\end{equation}

If $P^-_x$ has a greatest element~$q_n$, then $z_x$ will be the vertex~$z(q_n)$; if not, it will be an end. But this end need not be~$z(r)$. And if $z_x\ne z(r)$, we shall have a problem: to avoid a contradiction between \eqref{A+} and~\eqref{Ax}, we shall have to ensure that $t^-_x\ne t^+$ (which does not follow from the assumptions we have made about~$F$ so far), and define $F(x,t)$ so as to move $z_x$ to~$z(r)$ in the time interval $[t_x^-, t^+]$. Let us call our critical partition \emph{bad} if $z_x\ne z(r)$ for some critical point~$x$, which we then also call \emph{bad}.

In general, $\P$~may have uncountably many critical partitions, and one of these can have bad points~$x$ with infinitely many different~$z_x$. It will be crucial for our construction of~$F$, therefore, to prove that there can be only countably many {\em bad} partitions. For each of these, we shall be able to deal with all its bad points simultaneously.

For our proof that there are only countably many bad partitions, let us show first that
\begin{txteq}\label{MiTail}
   $P^-_{i(r)}\ne\es$ for every bad partition $r = (P^-,P^+)$, indeed for every critical partition that has a critical point.
\end{txteq}
To prove~\eqref{MiTail} let $i=i(r)$, let $x$ be a critical point for~$r$, and consider any $q\in P_x^-$. If $q\in M_i$ we are done. If not, then $q\in M_j$ for some $j>i$, because there are $p\in P^+_i$ with $q\prec p$. By the maximality of~$M_i$, this means that $q\not\le p$ for some $p\in M_i$. Since $q\le p$ for all $p\in P^+_i$ (as shown earlier), it follows that $P^-_i = M_i\sm P^+_i\ne\es$, as claimed.

By~\eqref{MiTail}, the nested union
 $$\big(a^-,b^-\big)
    := \bigcup_{p\in P^-_{i(r)}} \big( a(p),b(p)\big )$$
is not empty. As $q\le p$ for all $q\in P^-_i$ and~$p\in P^+_i$, we have $[a^-,b^-]\sub [a^+,b^+]$:
 $$a^+ \le a^-\quad\emtext{and}\quad b^- \le b^+.$$

Assuming that $r$ is bad, let us show that at least one of these inequalities is strict. Pick a bad point $x\in [a^+,b^+]$. Assume first that $x\in (a^-,b^-)$. Then $x$ lies in $(a(p),b(p))$ for some, and hence for all large enough, $p\in P^-_{i(r)}$. But then all these $p$ lie in~$P_x^-$, so $(a^-,b^-) \sub (a_x^-,b_x^-)$ and hence $a_x^-\le a^-$. (In fact, we have equality.) Recall that $a^+\le a_x^-$. Since $x$ is a bad point, we have $\sigma(a_x^-)\ne \sigma(a^+)$ and therefore $a_x^-\ne a^+$. Hence $a^+ < a_x^- \le a^-$, as desired.

On the other hand if $x\notin (a^-,b^-)$, then $x\in [a^+,a^-]\cup [b^-,b^+]$. We assume that $x\in [a^+,a^-]$ and show $a^+ < a^-$; the case of $x\in [b^-,b^+]$ is analogous, showing $b^- < b^+$. Pick $q\in P_x^-$. Then $x\in (a(q),b(q))\sub(a(p),b(p))$ for every $p\in P^+_{i(r)}$, so $a^+ \le a(q) < x\le a^-$ by the definition of~$a^+$.

We have thus shown that, for every bad partition~$r$, the set
 $$D(r) := (a^+,a^-)\cup (b^-,b^+)$$
is non-empty. We next show that these sets are disjoint for distinct $r = (P^-,P^+)$ and $\tilde r = (\tilde P^-,\tilde P^+)$ with the same index~$i$. As ${r\ne \tilde r}$, we may assume that there is a pair $p\in P^- \cap \tilde P^+$. Then $p\in M_i$, since $M_j\sub P^-\cap \tilde P^-$ for every~$j>i$, while $M_j\sub P^+\cap \tilde P^+$ for every~$j<i$. But then
 $$a^-\le a(p) < \tilde a^+ < \tilde b^+ < b(p)\le b^-$$
with the obvious notation. Since $D(r)\cap [a^-,b^-]=\es$ while $D(\tilde r)\sub (\tilde a^+,\tilde b^+)$, we have $D(r)\cap D(\tilde r)=\es$ as claimed.

As there are only countably many partition indices~$i$, and for every bad partition $r$ with index $i$ the set $D(r)$ contains a rational, this completes our proof that there are only countably many bad partitions.

\medbreak

Denote the set of all critical and all bad partitions by $\R$ and~$\R'$, respectively. Given $r\in\R$, write $a(r)$ and $b(r)$ for the points $a^+,b^+\in [0,1]$ defined above. Taking limits in~\eqref{imageofabp}, we obtain:
\begin{equation}\label{imageofabr}
  \text{\em For every $r\in\R$ we have $\sigma(a(r))=\sigma(b(r))\
     \big(=z(r)\in\Omega(G)\sub T\big)$.}
\end{equation}

Our plan is to begin the construction of~$F$ by extending our linear ordering $\preceq$ to $\P\cup\R$. We shall then choose disjoint time intervals $[s(q),t(q)]$ for all $q\in \P\cup\R'$, extending~\eqref{crudeorder} to~$\P\cup\R'$. For every $r = (P^-,P^+)\in\R$, the choices made for $\P$ will define times
 $$t_r^+\assign\inf\,\{\,t(p)\mid p\in P^+\}\quad\emtext{and}\quad
   t_r^-\assign\sup\,\{\,t(p)\mid p\in P^-\}\,.$$
In order to satisfy~\eqref{A+}, we shall have to have $f_{t_r^+}$ map all of $[a(r),b(r)]$ to~$\{z(r)\}$. For bad~$r$, this will require the insertion of a local homotopy $f_{s(r)}\to f_{t(r)}$ analogous to~\eqref{localhomotopies}, where $s(r)$ and $t(r)$ are chosen so that 
 $$t_r^-\le s(r) < t(r)\le t_r^+.$$

To extend $\preceq$ from $\P$ to $\P\cup\R$, we simply place all the partitions $r\in\R$ at their natural positions in the chains~$M_{i(r)}$. Indeed, let us extend our partial ordering $\le$ on~$\P$ to $\P\cup\R$ by letting $q\le q'$ whenever $[a(q),b(q)]\sub [a(q'),b(q')]$. Then
 $$\hat M_i := M_i\cup\{\,r\in\R\mid i(r) = i\,\}$$
is easily seen to be a $\le$-chain. We now define $\preceq$ on $\P\cup\R$ as we did on~$\P$: given $q,q'$, we put $q\preceq q'$ if either
\begin{itemize}
\item
  $q$, $q'$ lie in the same $\hat M_i$ and $q\le q'$;
 or\item
  $q\in \hat M_j$ and $q'\in \hat M_i$ with $i<j$.
\end{itemize}

\penalty-99

Let us note a couple of facts about this ordering. The fact that pairs are either nested or disjoint extends at once:
\begin{txteq}\label{surroundsucceed}
  Given $q,q'\in \P\cup\R$ with $q\preceq q'$, either $q\le q'$ or the intervals $(a(q),b(q))$ and $(a(q'),b(q'))$ are disjoint.
\end{txteq}
The following statements can be satisfied by suitable $p\in P^-_{i(r)}$, which we recall is non-empty if $r$ has a critical point~\eqref{MiTail}:
\begin{txteq}\label{pbetweenr}
 For every $r\in\R$ that has a critical point, in particular for every~$r\in\R'$, there is a pair $p\in\P$ such that $p\le r$. Given any $r'\in\R$ with $r'\prec r$, this $p$ can be chosen so that $r'\prec p\prec r$.
\end{txteq}
It is not difficult to describe $\preceq$ directly, without reference to~$\le\,$:
\begin{txteq}\label{directorder}
A partition $(P^-,P^+)\in\R$ precedes all pairs in $P^+$ and succeeds all pairs in~$P^-$. Two partitions $r=(P^-,P^+)$ and $\tilde r=(\tilde P^-,\tilde P^+)$ satisfy $r\preceq\tilde r$ if and only if $\tilde P^+\sub P^+$.
\end{txteq}

Having defined~$\preceq$, we can now choose disjoint intervals $[s(q),t(q)]$ for all $q\in\P\cup\R'$, to satisfy
\begin{equation}\label{crudeorderALL}
  \emtext{$t(q)<s(q')$ whenever $q\prec q'$.}
\end{equation}
This can be done inductively, since $\P\cup\R'$ is countable.

We are finally ready to define our homotopy $F=(f_t)_{t\in[0,1]}$. We first define $f_t$ for all $t\in[0,1]$ outside the set
 $$C := \bigcup_{q\in\P\cup\R'} \big(s(q),t(q)\big).$$
\vskip-\medskipamount\noindent
Given such $t\in [0,1]\sm C$, let 
 $$Q^t := \{\,q\in\P\cup\R'\mid t(q)\le t\,\}\,.$$
For every $x\in[0,1]$, the set
 $$Q_x^t := \{\,q\in Q^t\mid x\in(a(q),b(q))\,\}$$
is a $\le$-chain, by~\eqref{surroundsucceed}. If $Q_x^t\ne\es$, we set
\begin{equation*}
  (a^t_x,b^t_x)\assign\bigcup_{q\in Q_x^t}(a(q),b(q))
\end{equation*}
   \vskip-\smallskipamount\noindent
and define $f_t(x)\assign\sigma(a^t_x) = \sigma(b^t_x) \in T$. (Recall~\eqref{imageofabp} and~\eqref{imageofabr}.)
If $Q_x^t = \es$, we call $x$ \emph{unchanged at time~$t$}, define $f_t(x)\assign\sigma(x)$, and put $a^t_x\assign x =: b^t_x$. Thus,
\begin{txteq}\label{ftx}
   for all $t\notin C$ and all~$x$:  $f_t(x) = \sigma(a^t_x) = \sigma(b^t_x)$.
\end{txteq}

We need to show that these functions $f_t$ are continuous. This will follow from the fact that $\sigma$ is continous, once we have shown the following:
\begin{txteq}\label{atx}
   For all $t\notin C$ and all $x$ with $Q^t_x\ne\es$, the function $f_t$ is constant on $(a^t_x,b^t_x)$ with value~$f_t(x)\in T$.
\end{txteq}
To prove~\eqref{atx}, pick $y\in (a^t_x,b^t_x)$. We show that $Q^t_y\ne\es$ and $(a^t_y,b^t_y)=(a^t_x,b^t_x)$; then $f_t(y)=f_t(x)\in T$ by definition of~$f_t$. By the choice of~$y$, there exists $q\in Q^t$ such that $(a(q),b(q))$ contains both $x$ and~$y$, giving $q\in Q^t_x\cap Q^t_y$. As $Q^t_x$ is a $\le$-chain, we have $q\le q'$ for all large enough $q'\in Q^t_x$. Since $y\in (a(q),b(q))\sub (a(q'),b(q'))$ implies $q'\in Q^t_y$, all large enough $q'\in Q^t_x$ lie in~$Q^t_y$, giving $(a^t_x,b^t_x)\sub (a^t_y,b^t_y)$. Likewise, $(a^t_y,b^t_y)\sub (a^t_x,b^t_x)$ and hence $(a^t_x,b^t_x)=(a^t_y,b^t_y)$. This completes the proof of~\eqref{atx}, and with it the proof that the $f_t$ defined so far are continuous.

We have just shown that $(a^t_y,b^t_y)=(a^t_x,b^t_x)$ for all $x$ and~$y$ with $y\in(a^t_x,b^t_x)$. Therefore, for any $x,y$ the intervals $(a^t_x,b^t_x)$ and $(a^t_y,b^t_y)$ are either identical or disjoint: if $(a^t_x,b^t_x)$ meets $(a^t_y,b^t_y)$, in a point $z$ say, we have $(a^t_x,b^t_x)=(a^t_z,b^t_z)=(a^t_y,b^t_y)$. An immediate consequence of this is the following:
\begin{txteq}\label{unchanged}
   For all $t\in[0,1]\sm C$ and $x\in [0,1]$, the points $a^t_x$ and $b^t_x$ are unchanged at time $t$.
\end{txteq}

It remains to define $f_t$ for $t\in (s(q),t(q))$ with $q\in\P\cup\R'$. Since these intervals are disjoint, $f_{s(q)}$~and $f_{t(q)}$ are already defined. For each~$q$, our aim is to define the functions $f_t$ with $t\in (s(q),t(q))$ as a homotopy between $f_{s(q)}$ and~$f_{t(q)}$. When $q\in\R'$, this homotopy should retract $f_{s(q)}\restr [a(q),b(q)]$ to~$z(q)$ in~$T$ by a direct application of~\eqref{pathsinT}. When $q\in\P$, our plan is to follow~\eqref{localhomotopies} and achieve the same result in two stages. We first wish to use \eqref{pathsinT} to retract $f_{s(q)}\restr [a^-(q),b^-(q)]$ to~$z^-(q)$ in~$T$. This should turn $f_{s(q)}\restr [a(q),b(q)]$ into a path consisting of the two passes of $q$ through~$e(q)$ at the beginning and end, and a constant path with image~$z^-(p)$ in the middle. We then wish to retract this path to~$z(q)$ in~$e(q)$.

In order to apply \eqref{pathsinT} and implement~\eqref{localhomotopies} as just outlined, we have to verify the following prerequisites:
\begin{itemize}
   \item that $f_{s(q)}$ maps both $a(q)$ and $b(q)$ to~$z(q)$;
   \item that $f_{t(q)}$ maps all of $[a(q),b(q)]$ to~$z(q)$;
   \item if $q\in\P$: that $f_{s(q)}$ agrees with $\sigma$ on $D(q) := (a(q),a^-(q))\cup (b^-(q),b(q))$ and maps $[a^-(q),b^-(q)]$ to~$T$;
   \item if $q\in\R'$: that $f_{s(q)}$ maps $[a(q),b(q)]$ to~$T$.
\end{itemize}

Let us prove the first statement, as well as the second statement for $a(q)$ and~$b(q)$. At times $s(q)$ and~$t(q)$, both $a(q)$ and $b(q)$ were unchanged, because they can lie in an interval $(a(q'),b(q'))$ only when $q < q'$ and hence $t(q) < t(q')$. Therefore $f_{s(q)}$ and $f_{t(q)}$ both map $a(q)$ and $b(q)$ to $\sigma(a(q)) = \sigma(b(q)) = z(q)$. (Recall~\eqref{imageofabp} and~\eqref{imageofabr}.)

Next, we prove the second statment for $x\in (a(q),b(q))$. At time~$t(q)$, none of these $x$ was unchanged, since $q\in Q^{t(q)}_x$ for all these~$x$. In fact, $q$~is the greatest element of each of the $\le$-chains $Q^{t(q)}_x$. Therefore $(a^{t(q)}_x,b^{t(q)}_x) = (a(q),b(q))$, and hence $f_{t(q)} (x) = \sigma(a(q)) = z(q)$ by~\eqref{ftx}.

To prove the first part of the third statement note that, if $q\in\P$, all the points in $D(q)$ are still unchanged at time~$s(q)$. Hence $f_{s(q)}$~agrees with $\sigma$ on~$D(q)$.\looseness=-1

To prove the rest of the third and the fourth statement, consider any point $x$ in~$[a(q),b(q)]$, but not in $D(q)$ if $q\in\P$. We have to show that $f_{s(q)}(x)\in T$. This follows from~\eqref{atx} if $x$ is not unchanged at time~$s(q)$, so assume that it is. Then $f_{s(q)}(x) = \sigma(x)$. If this point is not in~$T$, then $x$ is an inner point of the domain of a pass contained in some $p\le q$. If $p=q$ this means that $x\in D(q)$, contradicting our choice of~$x$. Hence $p < q$, and $t(p) < s(q)$ by~\eqref{crudeorderALL}. Thus $p\in Q^{s(q)}_x \ne\es$, contradicting our assumption that $x$ was unchanged at time~$s(q)$.

Having checked the prerequisites, we may now apply~\eqref{pathsinT} as outlined earlier to choose $f_t$ for all $t\in (s(q),t(q))$ as follows:
\begin{txteq}\label{localFr}
  For every $r\in\R'$, the paths $(f_t)_{t\in[s(r),t(r)]}$ form a homotopy retracting $f_{s(r)} \restr [a(r),b(r)]$ to~$z(r)$.
\end{txteq}
\begin{txteq}\label{localFp}
  For every $p\in\P$, the paths $(f_t)_{t\in[s(p),t(p)]}$ form a homotopy that first retracts $f_{s(p)} \restr [a^-(p),b^-(p)]$ to~$z^-(p)$ in~$T$ and then retracts the resulting path on $[a(p),b(p)]$ to $z(p)$ in~$e(p)$.
\end{txteq}\smallskip

For our later proof that $F$ is continuous, let us note an important property of the homotopies in \eqref{localFr} and~\eqref{localFp}. Let $U$ be a neighbourhood in $|G|$ of an end~$\omega$; then $U\cap H$ is a neighbourhood of $\omega$ in $H$. By Lemma~\ref{shortarcs} there is a basic open neighbourhood $\hat U\sub U$ of $\omega$ in $|G|$ (i.e.,\ $\hat U=\hat C(S,\omega)$ for some finite set $S$ of vertices) such that for any $x,y\in \hat U\cap T$, the arc $xTy$ is contained in $U$ (and thus in $U\cap H$). Let $S'$ be the set of neighbours of $S$ in $C(S,\omega)$, note that these are finitely many. Call $U'\assign \hat C(S',\omega)$ a \emph{core of~$U$ around~$\omega$}. If $U'$ contains a vertex~$z(p)$, then $\hat U$ contains its neighbour~$z^-(p)$ and the edge~$e(p)$. The next statement therefore follows from the fact that the homotopies used in \eqref{localFr} and~\eqref{localFp} either run inside $e(p)$ or else are time-injective (by our definition of retracting).
\begin{txteq}\label{convex}
Let $q\in\P\cap\R'$ and $x\in (a(q),b(q))$, and let $U'\sub |G|$ be a core of a neighbourhood $U$ around an end. If both $f_{s(q)}(x)$ and $f_{t(q)}(x)$ lie in~$U'$, then $f_t(x)\in U$ for all $t\in[s(q),t(q)]$.
\end{txteq}
 
The definition of~$F$ is now complete. Note finally that
\begin{equation}\label{imageofF}
  \im F=\im \sigma\,.
\end{equation}
This completes part four of our proof.

\medbreak

It remains to show that the set $(f_t)_{t\in[0,1]}$ of paths in $H$ is a homotopy between $\sigma$ and~$\tau$. Since $Q^0=\es$, we have $f_0=\sigma$. Rather than proving that $f_1=\tau$, let us show that $\im f_1\sub T$; then \eqref{pathsinT} and our assumption that $\im\tau\sub T$ imply $f_1\sim\tau$, which is good enough. For a proof that $\im f_1\sub T$ consider any $x\in[0,1]$. If $x$ is not unchanged at time $t=1$, then $f_1(x)\in T$ by~\eqref{atx}. If it is, then $f_1(x)=\sigma(x)$. Since $\sigma$ is straight, we have $f_1(x)\in T$ unless $x$ lies in the interior of the domain of a pass of $\sigma$. But this pass is contained in some pair $p\in Q^1_x$, so $x$ was not unchanged at time~$1$, contradiction.

We now prove that $F$ is continuous. Let $x,t\in[0,1]$ be given, and let $U$ be any neighbourhood of $F(x,t)$ in~$|G|$; then $U\cap H$ is a neighbourhood of $F(x,t)$ in~$H$. If $F(x,t)$ is an end, let $U'$ be a core of $U$ around that end. We shall find an $\eps>0$ such that $F((x-\eps,x+\eps), (t-\eps,t+\eps))\subset U$, proceeding in two steps.

\begin{enumerate}[\rm 1.]
\item\label{frombelow}
We find an $\eps$ for which $F((x-\eps,x+\eps),(t-\eps,t])\subset U$.

For every $\eps>0$, let
 $$Q_{\eps} := \big\{\,q\in\P\cup\R': (t-\eps,t]\cap (s(q),t(q))\ne\es\,\big\}\,.$$

If $Q_{\eps}=\es$ for some~$\eps$, then for all $t'\in(t-\eps,t]$ we have $Q^{t'} = Q^t$ and hence $f_{t'}=f_t$. As $f_t$ is continuous, there is an $\eps_0 < \eps$ such that $F((x-\eps_0,x+\eps_0),\penalty-100 (t-\eps_0,t])\subset U$.

If $Q_{\eps}$ is never empty but finite for some~$\eps$, there exists $q\in\P\cup\R'$ such that $t\in(s(q),t(q)]$. Since $(f_t)_{t\in [s(q),t(q)]}$ was defined as a homotopy, in \eqref{localFr} or~\eqref{localFp}, we have $F((x-\eps_0,x+\eps_0),(t-\eps_0,t])\sub U$ for small enough $\eps_0<\eps$.

We may thus assume that $Q_{\eps}$ is infinite for every $\eps>0$, so it has no maximal element with respect to~$\preceq$. By the definition of $\preceq$, we can choose $\eps_0$ small enough that all pairs in $Q_{\eps_0}$ lie in the same chain~$M_i$. Then $i=i(r)$ for all partitions~$r$ in~$Q_{\eps_0}$, e.g.\ by~\eqref{MiTail}. Hence $Q_{\eps_0}\sub\hat M_i$, so the intervals $(a(q),b(q))$ with $q\in Q_{\eps_0}$ are nested; put
 $$(a,b) := \bigcup_{q\in Q_{\eps_0}} (a(q),b(q))\,.$$
   \vskip-\medskipamount
From \eqref{imageofabp}, \eqref{imageofabr} and Lemma~\ref{pass} we know that $\sigma(a)$ is the limit of an infinite sequence of ends or distinct vertices $\sigma(a(q)) = z(q)$, so $\sigma(a)$ must be an end. 

Our assumption that $Q_{\eps_0}$ has no maximal element also implies that $t\notin C$, so $F(x,t) = f_t(x)=\sigma(a^t_x)$ by~\eqref{ftx}. Moreover, there are points $t'\notin C$ arbitrarily close below~$t$.
If $F(x,t)$ is an end, it will suffice to find an $\eps < \eps_0$ such that ${t-\eps\notin C}$ and $F(x',t')\in U'$ for all $x'\in (x-\eps,x+\eps)$ and all $t'\in [t-\eps,t]\sm C$: then for all such $x'$ and all $t''\in (t-\eps,t)\cap C$ we shall have $F(x',t'')\in U$ by~\eqref{convex}.

 We distinguish two cases: that $x$ lies in~$(a,b)$ or not.

\begin{enumerate}[\rm (i)]
\item\label{meetingxbelow}
Suppose that $x\in (a,b)$. Then $a^t_x = a$, so \eqref{ftx} yields $F(x,t) = \sigma(a)$, which we know is an end. Since $\sigma$ is continuous, there is a $\delta>0$ such that $\sigma$ maps $[a,a+\delta)$ to~$U'$. Choose $\eps_1<\eps_0$ so that for all $q\in Q_{\eps_1}$ we have $a(q)\in [a,a+\delta)$ and  $(x-\eps_1,x+\eps_1)\subset(a(q),b(q))$.

Pick $q_0\in Q_{\eps_1}$. Choose $\eps_2 < \eps_1$ small enough that $t-\eps_2 > t(q_0)$, and so that $t-\eps_2\notin C$. Then for all $x'\in (x-\eps_2,x+\eps_2)$ and $t'\in [t-\eps_2,t]\sm C$ we have $q_0\in Q^{t'}_{x'}\ne\es$, and \eqref{ftx} yields $F(x',t') = \sigma(a^{t'}_{x'})$. These points lie in~$U'$, since $a\le a^{t'}_{x'}\le a(q_0)< a+\delta$.

\item\label{avoidingxbelow}
Suppose now that $x\notin(a,b)$, say $x\le a$. Note that for all $x'\notin (a,b)$ and $t'\in (t-\eps_0,t]\sm C$ we have $Q^{t'}_{x'} = Q^t_{x'}$, and hence $f_{t'}(x') = f_t(x')$.

If $x=a$, pick $q_0\in Q_{\eps_0}$. Then $f_{t(q_0)} (x) = f_t(x)\in U'$. (Note that as $a=a^t_y$ for all $y\in (a,b)$, we have $f_t(x)=f_t(a)=\sigma(a)$ by~\eqref{unchanged}, hence $f_t(x)$ is an end.) As $f_{t(q_0)}$ is continuous, there is an $\eps_1 \le t-t(q_0)\ (\le\eps_0)$ such that $t-\eps_1\notin C$ and ${f_{t(q_0)}((x-\eps_1, x+\eps_1))}\subset U'$. We show that $F(x',t')\in U'$ for all $x'\in (x-\eps_1, x+\eps_1)$ and $t'\in [t-\eps_1,t]\sm C$. If $Q^{t'}_{x'}=Q^{t(q_0)}_{x'}$, then $f_{t'}({x'})=f_{t(q_0)}({x'})\in U'$. Otherwise $Q^{t'}_{x'}\supsetneq Q^{t(q_0)}_{x'}$ (since $t(q_0) < t'$), and hence ${x'}\in (a,b)$. Then
  $$x-\eps_1 < a \le a^{t'}_{x'} < {x'} < x+\eps_1\,.$$
As $a^{t'}_{x'}$ is unchanged at time~$t'$~\eqref{unchanged} and hence also at time $t(q_0) < t'$, we have $f_{t'}({x'}) = \sigma(a^{t'}_{x'}) = f_{t(q_0)}(a^{t'}_{x'})$ by~\eqref{ftx}. This last point lies in~$U'$, by the above inequality and the choice of~$\eps_1$.

If $x<a$ then, as $f_t$ is continuous, there is an ${\eps_1} < \eps_0$ such that $x+{\eps_1} < a$ and $F(x',t)\in U$ for all $x'\in(x-{\eps_1},x+{\eps_1})$. Choose ${\eps_1}$ so that $t-{\eps_1}\notin C$. And as noted earlier, $F(x',t') = F(x',t)$ for every $t'\in[t-{\eps_1},t]\sm C$ and all these~$x'$, so $F(x',t')\in U$. On the other hand for $t'\in (t-{\eps_1},t) \cap C$, say $t'\in (s(q),t(q))$ with $q\in Q_{\eps_1}$, we have $f_{t'} (x') = f_{t(q)} (x')$ for these same~$x'$, because $x'\notin (a,b)\supseteq (a(q),b(q))$ and hence $f_{t'}(x')$ remained constant throughout the homotopy defined in \eqref{localFr} or~\eqref{localFp}. Since $t(q)\notin C$, we are thus home by the case of $t'\in[t-{\eps_1},t]\sm C$.
  \end{enumerate}

\item\label{fromabove}
  We find an $\eps$ for which $F((x-\eps,x+\eps),[t,t+\eps))\subset U$.

For every $\eps>0$, let
  $$Q_{\eps} := \big\{ q\in\P\cup\R': [t,t+\eps)\cap (s(q),t(q))\ne\es\big\}.$$
As in the first step, we may assume that $Q_{\eps}$ is infinite for every~$\eps$. Thus ${Q_{\eps}\ne\es}$, but $Q_{\eps}$ has no least element in~$\preceq$. 
Then $t\notin C$, so the sets
 $$P^+\assign\{\,p\in\P \mid s(p)\ge t\,\}\quad\text{and}\quad
       P^-\assign\{\,p\in\P \mid t(p)\le t\,\}$$
partition~$\P$. From~\eqref{pbetweenr} we know that $P^+$ meets every~$Q_{\eps}$. So $P^+$ has no least element in~$\preceq$, and
  \begin{equation}\label{earlypairs}
    t=\inf \{\,t(p) \mid p\in P^+\}.
  \end{equation}
Let us write $r\assign (P^-,P^+)$. But note that $P^-$ may be empty, in which case $r\notin\R$ and every $Q_\eps$ might meet infinitely many chains~$M_i$.

If $F(x,t)$ is an end then, as in Step~\ref{frombelow}, it will suffice to find an $\eps < \eps_0$ such that $t+\eps\notin C$ and $F(x',t')\in U'$ for all $x'\in (x-\eps,x+\eps)$ and all $t'\in [t,t+\eps]\sm C$.

We distinguish two cases.
  \begin{enumerate}
  \item\label{meetingxabove}
Our first case is that for every $\eps$ there is a $q\in Q_{\eps}$ with $x\in(a(q),b(q))$. Depending on whether $P^-$ is empty or not, we shall in two different ways define an end $z(r)$ and an interval $[a(r),b(r)]$ containing~$x$, and in each case prove that
 \begin{equation}\label{zr}
 f_t \emtext{ maps } (a(r),b(r)) \emtext{ to } z(r).
 \end{equation}
We shall then use~\eqref{zr} to find the $\eps$ desired in Step~\ref{fromabove}.

We first assume that $P^-\not=\emptyset$. Then $r$ is a critical partition, so $z(r)$ is defined and is an end. Every $Q_\eps\cap\P$ meets only finitely many chains~$M_i$, and for the largest of these~$i$ we have ${Q_{\eps_0}\sub\hat M_i}$ for some small enough~$\eps_0$. Then the intervals $[a(q),b(q)]$ with $q\in Q_{\eps_0}$ are nested, and
  $$[a(r),b(r)] = \bigcap_{p\in P^+_i} [a(p),b(p)]
     = \bigcap_{q\in Q_{\eps_0}} [a(q),b(q)]$$
by definition of $a(r)$ and~$b(r)$, and~\eqref{pbetweenr}. By our assumption for Case~\ref{meetingxabove}, $x\in [a(r),b(r)]$.

If $r$ is bad, then $r$ lies in $\R'$ and precedes all pairs in $P^+$. By~\eqref{crudeorderALL} and~\eqref{earlypairs} we have $t(r)\le t$, so $r\in Q^t$. In fact, $r$~is the greatest element of~$Q^t$. For $r$ succeeds every $p\in Q^t\cap\P$, since these $p$ lie in~$P^-$. But then $r$ also succeeds any $r'\in Q^t\cap\R'$: otherwise $r\prec p\prec r'$ for some~$p\in P^+$ by~\eqref{pbetweenr}, while $t(p)\le t(r')\le t$ (by~\eqref{crudeorderALL} and $r'\in Q^t$) implies that $p\in P^-$. Now as $r$ is the greatest element of~$Q^t$, it is also the greatest element of~$Q^t_y$ for every $y\in (a(r),b(r))$, giving $a^t_y = a(r)$. As $t\notin C$, \eqref{ftx} yields
 $$f_t(y) = \sigma(a^t_y) = \sigma(a(r)) \specrel={\eqref{A+}} z(r)\,,$$
completing the proof of~\eqref{zr} for the case that $P^-\ne\es$ and $r$ is bad.
    
Let us suppose now that $r$ is not bad (so $r\in\R\sm\R'$), and once more show that $f_t(y) = z(r)$ for every $y\in (a(r),b(r))$. As before, we have $f_t(y) = \sigma(a^t_y)$ by~\eqref{ftx}. If $y$ is critical, then $Q^t_y\ne\es$, and hence $a^t_y = a^-_y$ by~\eqref{pbetweenr} and the definitions of $r$, $a^-_y$ and~$a^t_y$. Thus
 $$f_t(y) = \sigma(a^t_y) = \sigma(a^-_y) = z_y = z(r)$$
since $y$ is not bad. Now assume that $y$ is not critical. By definition of~$r$, this means that
\begin{equation}\label{Qty}
Q^t_y\cap\P = \es\,.
\end{equation}

We first prove that $y$ is unchanged at time~$t$. Indeed, otherwise $Q^t_y\ne\es$, and by~\eqref{Qty} there exists an $\tilde r=(\tilde P^-,\tilde P^+)\in\R'$ such that $y\in (a(\tilde r),b(\tilde r))$ and $t(\tilde r)\le t$. By~\eqref{crudeorderALL}, $\tilde r$~precedes all pairs in $P^+$, which by \eqref{directorder} implies that $\tilde P^+\supset P^+$. As $\tilde r\not= r$ (since $\tilde r\in\R'$ but $r\notin\R'$), there exists a pair $p\in\tilde P^+\cap P^-$. As all sufficiently early pairs of $\tilde P^+$ lie in~$M_{i(\tilde r)}$, we can find this $p$ in~$M_{i(\tilde r)}$, giving $\tilde r\le p$. But then  $y\in (a(\tilde r),b(\tilde r))\sub (a(p),b(p))$ and $t(p)\le t$, contradicting~\eqref{Qty}.

Thus $y$ is unchanged at time~$t$. Then~\eqref{ftx} yields $f_t(y) =\sigma(y)$, so let us show that $\sigma(y)=z(r)$. Our aim is to prove $\sigma(y)=z(r)$ using~\eqref{imageofsigma}. Let $[a,b]\ni y$ be a maximal interval with the property that $(a,b)$ (which is allowed to be empty) avoids every domain of a pass of $\sigma$. As every neighbourhood of $a(r)$ or $b(r)$ meets the domain of a pass of $\sigma$ (namely, in $a(p)$ or $b(p)$ for every sufficiently small $p\in P^+_{i(r)}$) we have $[a,b]\sub[a(r),b(r)]$. We first prove that every point in $[a,b]$ is unchanged at time $t$. Indeed, otherwise there is a $q\in Q^t$ for which $(a(q),b(q))$ meets $[a,b]$. As $y$ is unchanged at time $t$ we have $(a(q),b(q))\not\supset[a,b]$ and thus $a(q)\in(a,b)$ or $b(q)\in(a,b)$, say $a(q)\in(a,b)$. If $q\in\P$, then $[a(q),a^-(q)]$ is the domain of a pass of $\sigma$ that meets $(a,b)$, a contradiction. If $q=(\tilde P^-,\tilde P^+)\in\R'$, then since $[a(q),b(q)]=\bigcap_{p\in \tilde P^+_{i(q)}}[a(p),b(p)]$ there is a pair $p\in \tilde P^+_{i(q)}$ with $a(p)\in(a,a(q))\sub (a,b)$, with a similar contradiction. Hence every point in $[a,b]$ is unchanged at time $t$ and thus $f_t \restr [a,b] = \sigma \restr [a,b]$.

Let us show that $\sigma(a)=\sigma(b)=z(r)$: then either $[a,b]=\{y\}$ and thus $\sigma(y)=z(r)$, or $\sigma$ maps all of~$[a,b]$, including~$y$, to $z(r)$ by~\eqref{imageofsigma}. We prove $\sigma(a)=z(r)$; the proof that $\sigma(b)=z(r)$ is analogous. If $a=a(r)$, then $\sigma(a) = \sigma(a(r)) = z(r)$, by definition of~$z(r)$. If $a\ne a(r)$ then, by the choice of $[a,b]$, there is a sequence $y_0<y_1<\ldots$ of points in $[a(r),b(r)]$ such that $\lim_n y_n=a$ and such that every $y_n$ lies in the interior of the domain of a pass of~$\sigma$ (possibly the same for all~$n$), and hence is critical. As $a$ is unchanged at time~$t$, and hence $a\notin (a(q),b(q))$ for every $q\in Q^t$, we have $y_n \le b^t_{y_n} \le a$ and hence also $\lim_n b^t_{y_n}=a$. As $\sigma$ is continuous, the points $\sigma(b^t_{y_n})$ converge in~$H$ to~$\sigma(a)$. Since each $y_n$ is critical but not bad, we have $\sigma(b^t_{y_n})=z(r)$ for every $n$, and thus $\sigma(a)=z(r)$. This completes the proof of~\eqref{zr} for the case of $P^-\ne\es$.

We now assume that $P^-=\es$. Then every $x\in [0,1]$ is unchanged at time~$t$, since $Q^t_x\ne\es$ would imply $P^-\ne\es$ by~\eqref{pbetweenr} and~\eqref{crudeorderALL}. Thus, $f_t = \sigma$.
Consider the $\le$-chain $Q^1_x$ of all $q\in\P\cup\R'$ with $x\in(a(q),b(q))$. By our assumption for Case~\ref{meetingxabove},
\begin{equation}\label{1xeps}
  Q^1_x\cap Q_\eps\ne\es \quad {\hbox{\em for all $\eps>0$.}}
\end{equation}
Since $Q^t_x=\es$ this means that~$Q^1_x$, like~$Q_\eps$, has no least element, and by~\eqref{pbetweenr} neither does $Q^1_x\cap\P = Q^1_x\cap P^+$. Therefore
\begin{equation*}
  [a(r),b(r)] \assign \bigcap_{p\in Q^1_x\cap P^+}[a(p),b(p)] = 
	 \bigcap_{q\in Q^1_x}[a(q),b(q)]\,.
\end{equation*}
Pick $p_0,p_1,\dotsc\in Q^1_x\cap P^+$ with $\lim_n a(p_n)=a(r)$. As $\sigma$ is continuous, $\lim_n z(p_n)=\sigma(a(r))=:z(r)$. (Recall that $z(p_n)=\sigma(a(p_n))$.) By~\eqref{imageofabp} and Lemma~\ref{pass}, $z(r)$ is an end.

To complete the proof of~\eqref{zr}, we show that the interval $(a(r),b(r))$ avoids every domain of a pass of~$\sigma$: then it is maximal with this property, and \eqref{zr} follows from~\eqref{imageofsigma} and the fact that $f_t=\sigma$. (For the application of~\eqref{imageofsigma} note that $\sigma(b(r))=\sigma(a(r))$, by taking limits in~\eqref{imageofabp}.) Suppose that $(a(r),b(r))$ contains a point $y$ from the domain of a pass of~$\sigma$, say of the pair~$p$. If $p$ was an element of~$Q^1_x$, there would be another pair $p'\le p$ in~$Q^1_x$, with $y\notin (a(p'),b(p'))$. As this contradicts the choice of $y$ as a point in~$(a(r),b(r))$, we have $p\notin Q^1_x$. Thus, $x\notin (a(p),b(p))$, and hence $p\not\ge q$ for all $q\in Q^1_x$. But $p$ is nested with every such~$q$, since $y\in (a(p),b(p))\cap (a(q),b(q))$. Therefore $p < q$, and hence
 $$t < t(p) < s(q) \quad {\hbox{for all $q\in Q^1_x\,$,}}$$
by~\eqref{crudeorderALL} and since $p\notin P^-=\es$. But then for $\eps < t(p)-t$ we have $Q_{\eps}\cap Q^1_x=\es$, contradicting~\eqref{1xeps}. This completes the proof of~\eqref{zr}.

We have thus shown for both sets of definitions that $f_t$ maps $(a(r),b(r))$, and hence also $[a(r),b(r)]$, to the end~$z(r)$. As $x\in [a(r),b(r)]$, we in particular have $z(r) = f_t(x)\in U'$. As $f_t$ is continuous, there is a $\delta > 0$ such that $f_t$ maps $(a(r)-\delta,b(r)+\delta)$ to~$U'$. By the definition of $[a(r),b(r)]$ (in either case) we can find a pair $p_0\in P^+$ such that $a(p_0) \in (a(r)-\delta,a(r))$ and $b(p_0)\in(b(r),b(r)+\delta)$. Choose $\eps$ so that $t+\eps\notin C$, and small enough that $t+\eps < t(p_0)$ as well as $(x-\eps,x+\eps) \subset (a(p_0),b(p_0))$. Then for all $t'\in[t,t+\eps]\sm C$ and $x'\in(x-\eps,x+\eps)$ we have $p_0\in Q^{t(p_0)}_{x'}\supseteq Q^{t'}_{x'}$ and hence $a(p_0) = a^{t(p_0)}_{x'}\le a^{t'}_{x'}$. Thus,
    \begin{equation*}
      a(r)-\delta < a(p_0) \le a^{t'}_{x'} \le x' < x+\eps < b(r)+\delta\,,
    \end{equation*}
giving $f_t(a^{t'}_{x'}) \in U'$ by the choice of~$\delta$. But $a^{t'}_{x'}$ is unchanged at time~$t'$~\eqref{unchanged}, and hence also at time~$t\le t'$. So this latter point is just~$\sigma(a^{t'}_{x'})$, giving $F(x',t')= \sigma(a^{t'}_{x'}) = f_t(a^{t'}_{x'})\in U'$ by~\eqref{ftx}.

  \item\label{avoidingxabove}
Our second case is that there is an $\eps_0$ such that $x\notin(a(q),b(q))$ for all $q\in Q_{\eps_0}$. Suppose first that there is even an $\eps_1 < \eps_0$ such that $(x-\eps_1,x+\eps_1)$ avoids $(a(q),b(q))$ for all $q\in Q_{\eps_1}$. Then consider any $x'\in(x-\eps_1,x+\eps_1)$. For every $t'\in[t,t+\eps_1)\sm C$ we have $F(x',t')=F(x',t)$ by~\eqref{ftx}, since $Q^{t'}_{x'}=Q^t_{x'}$. For $t'\in[t,t+\eps_1)\cap C$, say $t'\in (s(q),t(q))$ with $q\in Q_{\eps_1}$, this implies $F(x',t') = F(x',s(q)) = F(x',t)$ by \eqref{localFr} or~\eqref{localFp}, since $x'\notin (a(q),b(q))$ and $s(q)\in[t,t+\eps_1)\sm C$. As $f_t$ is continuous, there is an $\eps_2 < \eps_1$ such that $f_t$ maps $(x-\eps_2,x+\eps_2)$ to~$U$. Then $F(x',t') = F(x',t)\in U$ for all $x'\in (x-\eps_2,x+\eps_2)$ and $t'\in [t,t+\eps_2)$.

We may thus assume that there is no such~$\eps_1$. Then:
\begin{txteq}\label{intmeets}    
For every~$\eps$, the interval $(x-\eps,x+\eps)$ meets $(a(q),b(q))$ for infinitely many $q\in Q_{\eps}$.
\end{txteq}

By~\eqref{intmeets} there is a sequence $q_0,q_1,\dotsc$ of pairs and partitions in $Q_{\eps_0}$ with $\lim_n a(q_n)=x$ or $\lim_n b(q_n)=x$; we assume that $\lim_n a(q_n)=x$. As $\sigma$ is continuous, we have $\lim_n \sigma(a_n)=\sigma(x)$. By~\eqref{imageofabp} and~\eqref{imageofabr}, the sequence $(\sigma(a_n))_{n\in\N}$ is a sequence of vertices and ends, and every vertex appears only finitely often (Lemma~\ref{pass}). Therefore $\sigma(x)$ is an end.

Let us show that $x$ is unchanged at time~$t$. By~\eqref{intmeets}, any interval $(a(q),b(q))$ that contains $x$ meets some $(a(q_n),b(q_n))$, but is not contained in it since $q_n\in Q_{\eps_0}$ and hence $x\in(a(q),b(q))\sm(a(q_n),b(q_n))$. Thus $q>q_n$ by~\eqref{surroundsucceed}, and $t(q)>t(q_n)>t$ by~\eqref{crudeorderALL}, giving $q\notin Q^t_x$ as desired.
As $x\notin (a(q),b(q))$ for every $q\in Q_{\eps_0}$, $x$~remains unchanged at all times $t' \in [t,t+\eps_0]\sm C$.

As $x$ is unchanged at time~$t$, we have $\sigma(x) = F(x,t)\in U'$. As $\sigma$ is continuous, there is an $\eps < \eps_0$ such that $\sigma((x-\eps,x+\eps)) \subset U'$ and $t+\eps \notin C$. Then for every $x'\in[x,x+\eps)$ and $t'\in[t,t+\eps]\setminus C$ we have $x \le a^{t'}_{x'}$ ($\le x'$), because $x$ is unchanged at time $t'$ and hence $x\notin (a(q),b(q))$ for every $q\in Q^{t'}_{x'}$. Likewise, for $x'\in(x-\eps,x]$ we have $x' \le b^{t'}_{x'} \le x$. Thus for every $x'\in(x-\eps,x+\eps)$ and $t'\in[t,t+\eps]\setminus C$ we have $a^{t'}_{x'}\in(x-\eps,x+\eps)$ or $b^{t'}_{x'}\in(x-\eps,x+\eps)$, say $b^{t'}_{x'}\in(x-\eps,x+\eps)$. By~\eqref{unchanged}, $b^{t'}_{x'}$ is unchanged at time~$t'$ and hence also at time $t\le t'$. We thus have $F(x',t') = \sigma(a^{t'}_{x'}) = \sigma(b^{t'}_{x'})=f_t(b^{t'}_{x'}) \in U'$ by~\eqref{ftx}.
  \end{enumerate}
\end{enumerate}

\noindent
This completes the fifth and final part of our proof that $\sigma$ and $\tau$ are homtopic if $\tau$ does not traverse chords.

Before we consider the general case where $\tau$ too travserses chords, let us remind ourselves in which subsets of~$H$ the homotopies considered so far run. In part one of the above proof, we first used~\eqref{straight} to straighten $\sigma$ to a path~$\sigma'$, which we then trimmed further to obtain a path $\sigma''$ satisfying~\eqref{imageofsigma}. This path $\sigma''$ served as $f_0$ for our homotopy~$F$, which ended with a path~$f_1$ in $T$. This path was homotopic to the straightened version $\tau'$ of the original path~$\tau$ (not traversing any chords). We thus found homotopies
 $$\sigma \sim \sigma' \sim f_0 \sim f_1 \sim \tau' \sim \tau\,.$$
The first of these homotopies runs in~$\im\sigma$; see~\eqref{straight}. In the second homotopy we retracted segments $\sigma'\restr[a,b]\sub T$ to paths $\sigma''\restr[a,b]$ with image~$\sigma'(a)T\sigma'(b)$. This is the unique $\sigma'(a)$--$\sigma'(b)$ arc in~$T$, so Lemma~\ref{arc} implies that the image of $\sigma'\restr[a,b]$ contains it. The homotopy between $\sigma'\restr[a,b]$ and~$\sigma''\restr[a,b]$, which \eqref{pathsinT} says runs in the union of the images of those two paths, thus in fact runs in $\im\sigma'\sub\im\sigma$, and hence so does the entire second homotopy $\sigma'\sim\sigma'' = f_0$. The third homotopy, $f_0\sim f_1$, runs in $\im f_0\sub\im\sigma'\sub\im\sigma$ by~\eqref{imageofF}. Similarly, the last homotopy $\tau\sim\tau'$ runs in~$\im\tau$, and the penultimate one, $f_1\sim\tau'$, runs in $\im f_1\cup \im\tau' \sub \im\sigma \cup \im\tau$. All in all, we have shown the following:
\begin{txteq}\label{smallhomotopy}
  If $\tau$ traverses no chords and $w_{\sigma}$ reduces to the empty word, then there is a homotopy in $\im\sigma \cup \im\tau$ between $\sigma$ and $\tau$.
\end{txteq}

To complete our proof of Lemma~\ref{lemma:reducible}, we now consider the case in which both $\sigma$ and $\tau$ traverse chords. By Lemma~\ref{reducedpath}, there are paths $\sigma'$ and~$\tau'$ such that $w_{\sigma'}=r(w_{\sigma})$ and $w_{\tau'}=r(w_{\tau})$, and we may further assume that every pass of $\sigma'$ is also a pass of $\sigma$ while every pass of $\tau'$ is also a pass of~$\tau$. Statement~\eqref{smallhomotopy}, applied to every non-trivial interval $[a,b]$ that is maximal with the property that it avoids the interior of every domain of a pass of~$\sigma'$, yields $\sigma \sim \sigma'$ by Lemma~\ref{infhomotopies}: note that $\sigma\restr [a,b]$ and $\sigma'\restr [a,b]$ have the same first and last point (because $a$ and $b$ are either boundary points of domains of common passes of $\sigma$ and~$\sigma'$ or limits of such points, and $\sigma$ and $\sigma'$ are continuous), and the reduction of $w_\sigma$ to $w_{\sigma'}$ defines a reduction of $w_{\sigma\restriction [a,b]}$ to $w_{\sigma'\restriction [a,b]} = \es$. Likewise, we obtain $\tau \sim \tau'$. It thus suffices to prove $\sigma'\sim\tau'$, from our assumptions that
\begin{equation}\label{interim}
  w_{\sigma'} = r(w_\sigma) = r(w_\tau) = w_{\tau'}
\end{equation}
(cf.\ Lemma~\ref{uniquereduced}).

Our aim is to use~\eqref{samepasses} to obtain the desired homotopy $\sigma'\sim \tau'$. But~\eqref{samepasses} requires that the two paths considered have the same passes, not just the same trace. In order to make~\eqref{samepasses} applicable, we therefore have to `synchronize' corresponding passes of $\sigma'$ and~$\tau'$: make their domains coincide by reparametrizing $\sigma'$ and~$\tau'$, and make $\sigma'$ and~$\tau'$ agree on those domains by applying local homotopies inside the corrsponding chords.

Every path $\alpha\colon [0,1]\to H$ defines a partition of $[0,1]$ into intervals: the interiors of domains of passes, and the (closed) components of the rest of $[0,1]$. The set $\I_{\alpha}$ of all those intervals, including trivial `intervals' $[x,x] = \{x\}$, inherits a linear ordering from~$[0,1]$. The bijection between the passes of $\sigma'$ and $\tau'$ provided by~\eqref{interim} defines an order-preserving bijection $\pi\colon\I_{\sigma'}\to\I_{\tau'}$.

Although $\pi$ maps open to open and closed to closed intervals, it might map non-trivial closed intervals to trivial ones or vice versa. In order to synchronize $\sigma'$ with~$\tau'$ as planned, we therefore have to expand trivial closed intervals to non-trivial ones in our reparametrizations of $\sigma'$ and~$\tau'$. This will be possible, since clearly $\I_{\sigma'}$ and $\I_{\tau'}$ contain only countably many trivial intervals whose corresponding interval in the other set is non-trivial.

We may thus partition $[0,1]$ into a set $\I$ of intervals so that there exist order-preserving bijections $\pi_{\sigma'}\colon \I\to\I_{\sigma'}$ and ${\pi_{\tau'}\colon\I\to\I_{\tau'}}$ that map open to open intervals bijectively, and trivial to trivial intervals, and which commute with $\pi\colon\I_{\sigma'}\to\I_{\tau'}$. We can now define surjective maps $\phi,\psi\colon [0,1]\to [0,1]$ such that $\phi$ maps every $I\in\I$ onto $\pi_{\sigma'}(I)\in\I_{\sigma'}$ and $\psi$ maps every $I\in\I$ onto $\pi_{\tau'}(I)\in\I_{\tau'}$.

Clearly, $\sigma'':=\sigma'\circ\phi$ is homotopic to~$\sigma'$, and $\tau'':= \tau'\circ\psi$ is homotopic to~$\tau'$. So it suffices to show that $\sigma''\sim\tau''$. But these maps now have not only the same trace but also the same domains of corresponding passes. Combining homotopies between corresponding passes inside their respective chords with a homotopy between the rests of $\sigma''$ and~$\tau''$ as in~\eqref{samepasses} yields the desired homotopy $\sigma''\sim\tau''$, by Lemma~\ref{infhomotopies}.
\end{proof}

\bibliographystyle{amsplain}
\bibliography{collective}

\small
\parindent=0pt
\vskip2mm plus 1fill

\begin{tabular}{cc}
\begin{minipage}[t]{0.5\linewidth}
Reinhard Diestel\\
Mathematisches Seminar\\
Universit\"at Hamburg\\
Bundesstra\ss e 55\\
20146 Hamburg\\
Germany\\
\end{minipage} 
&
\begin{minipage}[t]{0.5\linewidth}
Philipp Spr\"ussel\\
Mathematisches Seminar\\
Universit\"at Hamburg\\
Bundesstra\ss e 55\\
20146 Hamburg\\
Germany\\
\end{minipage}
\end{tabular} 

\end{document}